\documentclass[12pt]{article}
\usepackage{amsmath}
\usepackage{amssymb}
\usepackage{amsthm}
\usepackage{verbatim}
\usepackage{mathrsfs}

\newcommand{\0}{{\mathcal O}}
\newcommand{\D}{{\mathcal D}}

\newcommand{\R}[0]{\mathbb R}

\newcommand{\Ds}[0]{\mathcal D}

\newtheorem{Th}{Theorem}[section]
\newtheorem{Lemma}{Lemma}[section]
\newtheorem{Prop}[Lemma]{Proposition}
\newtheorem{Coro}[Th]{Corollary}

\begin{document}

\title{The group of Symplectomorphisms of $\R^{2n}$ and the Euler equations}
\author{ Hasan \.{I}nci\\
Department of Mathematics, Ko\c{c} University\\
Rumelifeneri Yolu, 34450 Sar{\i}yer \.{I}stanbul T\"urkiye\\
        {\it email: } {hinci@ku.edu.tr}
}

\maketitle

\begin{abstract}
	In this paper we consider the ``symplectic'' version of the Euler equations studied by Ebin \cite{ebin}. We show that these equations are globally well-posed on the Sobolev space $H^s(\R^{2n})$ for $n \geq 1$ and $s > 2n/2+1$. The mechanism underlying global well-posedness has similarities to the case of the 2D Euler equations. Moreover we consider the group of symplectomorphisms $\Ds^s_\omega(\R^{2n})$ of Sobolev type $H^s$ preserving the symplectic form $\omega=dx_1 \wedge dx_2 + \ldots + dx_{2n-1} \wedge dx_{2n}$. We show that $\Ds^s_\omega(\R^{2n})$ is a closed analytic submanifold of the full group $\Ds^s(\R^{2n})$ of diffeomorphisms of Sobolev type $H^s$ preserving the orientation. We prove that the symplectic version of the Euler equations has a Lagrangian formulation on $\D^s_\omega(\R^{2n})$ as an analytic second order ODE in the manner of the Euler-Arnold formalism \cite{arnold}. In contrast to this ``smooth'' behaviour in Lagrangian coordinates we show that it has a very ``rough'' behaviour in Eulerian coordinates. To be precise we show that the time $T > 0$ solution map $u_0 \mapsto u(T)$ mapping the initial value of the solution to its time $T$ value is nowhere locally uniformly continuous. In particular the solution map is nowhere locally Lipschitz. 
\end{abstract}


\newpage

\section{Introduction}\label{section_introduction}

The initial value problem for the incompressible Euler equations on $\R^m,m \geq 2$, can be written as
\begin{equation}\label{euler}
	u_t+(u \cdot \nabla) u \in H_\mu^\perp,\; u \in H_\mu,\; u(0)=u_0 \in H_\mu,
\end{equation}
where $u=(u_1,\ldots,u_m):\R \times \R^m \to \R^m$ is the unknown time dependent velocity field of the fluid, $H_\mu$ is the space of divergence free vector fields on $\R^m$ and $H_\mu^\perp$ is its $L^2$ orthogonal complement. Here $\mu=dx_1 \wedge \ldots \wedge dx_m$ denotes the volume form on $\R^m$ and $H_\mu$ can be described as the space of vector fields $X$ on $\R^m$ s.t. $L_X \mu=0$, where $L_X$ is the Lie derivative. Usually one writes the first equation of \eqref{euler} in the form
\[
	u_t + (u \cdot \nabla) u = -\nabla p
\]
for an unknown function $p$. Note that the vector fields which are $L^2$ orthogonal to the space of divergence free vector fields are of the form $-\nabla p$. This follows from the Hodge decomposition (or in this special situation from the Helmholtz decomposition which precedes Hodge). In \cite{ebin} Ebin studied a ``symplectic'' modification of \eqref{euler} on a closed manifold $M^{2n}$. In $\R^{2n},n \geq 1$, this modification of \eqref{euler} reads as 
\begin{equation}\label{symplectic_euler}
	u_t+(u \cdot \nabla) u \in H_\omega^\perp,\; u \in H_\omega,\;u(0)=u_0 \in H_\omega,
\end{equation}
where $u=(u_1,\ldots,u_{2n})$ is the unknown time dependent velocity field, $H_\omega$ the space of symplectic vector fields on $\R^{2n}$ with respect to the symplectic form $\omega=dx_1 \wedge dx_2 + \ldots + dx_{2n-1} \wedge dx_{2n}$ and $H_\omega^\perp$ is its $L^2$ orthogonal complement. More precisely $H_\omega$ is the space of vector fields $X$ on $\R^{2n}$ s.t. $L_X \omega=0$. Note that for $\R^2$ the systems \eqref{euler} and \eqref{symplectic_euler} coincide.\\
We will prove

\begin{Th}\label{th_gwp}
	Let $n \geq 1$ and $s > 2n/2+1$. Then \eqref{symplectic_euler} is globally well-posed in the Sobolev space $H^s(\R^{2n})$.
\end{Th}

In \cite{ebin} Ebin showed the analog of Theorem \ref{th_gwp} for a closed manifold $M^{2n}$.\\
To consider \eqref{symplectic_euler} on $L^2$ based spaces is quite natural since the $L^2$ norm of a solution $u$ is conserved. One can see this by taking the $L^2$ inner product of the equation with $u$ and integrating. Recall that the Sobolev space on $\R^m,m \geq 1$, of class $s \geq 0$ can be defined as
\[
	H^s(\R^m)=\{f \in L^2(\R^m) \;|\; \|f\|_{H^s} < \infty \},
\]
where $\|f\|_{H^s}=\left(\int_{\R^m} (1+|\xi|^2)^s |\hat f(\xi)|^2 \;d\xi \right)^{1/2}$ and $\hat f$ is the Fourier Transform of $f$. For $s > m/2+k$ and $k \geq 0$ the Sobolev Imbedding Theorem gives $H^s(\R^m) \hookrightarrow C^k_0(\R^m)$. Here $C^k_0(\R^m)$ is the space of $C^k$ functions on $\R^m$ vanishing at infinity together with their derivatives up to order $k$. Moreover for $s > m/2$ and $0 \leq s' \leq s$ the multiplication
\[
	H^s(\R^m) \times H^{s'}(\R^m) \to H^{s'}(\R^m),\;(f,g) \mapsto f \cdot g
\]
is a bounded map. When we speak about a solution $u$ to \eqref{symplectic_euler} in $H^s(\R^{2n})$ for $s > 2n/2+1$ on $[0,T],T \geq 0$, then we mean  a $u \in C([0,T];H^s_\omega(\R^{2n};\R^{2n}))$ s.t. there is $Z \in C([0,T];H^{s-1}_\omega(\R^{2n};\R^{2n})^\perp)$ with
\[
	u(t)=\int_0^t Z(s)-(u(s) \cdot \nabla)u(s) \;ds,\;0 \leq t \leq T.
\]
Note that $(u \cdot \nabla)u \in H^{s-1}$. So it is natural to assume this regularity for $Z$. For basic concepts in Sobolev spaces see \cite{composition}.\\
Our goal is to write \eqref{symplectic_euler} in Lagrangian coordinates, i.e. in terms of the flow map $\varphi$ of $u$
\[
	\varphi_t(t,x)=u(t,\varphi(t,x)),\;\varphi(0,x)=x,\;t \geq 0,x \in \R^{2n}.
\]
The convenient functional space for Lagrangian coordinates $\varphi:\R^m \to \R^m$ of Sobolev type in the classical regime, i.e. when $\varphi$ is $C^1$, was studied in \cite{composition}. It reads for $m \geq 1$ and $s > m/2+1$ as
\[
	\Ds^s(\R^m)=\{\varphi:\R^m \to \R^m \;|\; \varphi-\text{id} \in H^s(\R^m;\R^m),\; \det(d_x \varphi) > 0, x \in \R^m \},
\]
where $\text{id}:\R^m \to \R^m$ is the identity map, $H^s(\R^m;\R^m)$ is the space of vector fields on $\R^m$ of Sobolev class $s$ and $d_x \varphi$ is the Jacobian matrix of $\varphi$ at $x \in \R^m$. Due to the Sobolev Imbedding $H^s(\R^m) \hookrightarrow C_0^1(\R^m)$ one sees that $\Ds^s(\R^m)$ consists of $C^1$ diffeomorphisms of $\R^m$ and that $\Ds^s(\R^m) - \text{id} \subset H^s(\R^m;\R^m)$ is an open set. So $\Ds^s(\R^m)$ has naturally a Hilbert manifold structure. In \cite{composition} it was shown that $\D^s(\R^m)$ is a topological group under composition of maps.\\
We introduce for $n \geq 1$, $s > 2n/2+1$ and the symplectic structure $\omega=dx_1 \wedge dx_2 + \ldots + dx_{2n-1} \wedge dx_{2n}$ on $\R^{2n}$ the subgroup $\Ds^s_\omega(\R^{2n}) \subset \Ds^s(\R^{2n})$ consisting of symplectomorphisms of $(\R^{2n},\omega)$, i.e.
\[
	\Ds^s_\omega(\R^{2n})=\{\varphi \in \Ds^s(\R^{2n}) \;|\; \varphi^\ast \omega=\omega \},
\]
where $\varphi^\ast$ is the pullback by $\varphi$. We will prove

\begin{Th}\label{th_submanifold}
Let $n \geq 1$ and $s > 2n/2+1$. Then $\Ds^s_\omega(\R^{2n})$ is a closed analytic submanifold of $\Ds^s(\R^{2n})$.
\end{Th}

The smooth analog of Theorem \ref{th_submanifold} for a compact manifold $M^{2n}$ was shown in \cite{ebin_marsden}. In \cite{lagrangian} it was shown for $m \geq 1$, $s > m/2+1$ and the volume form $\mu=dx_1 \wedge \ldots \wedge dx_m$ on $\R^m$ that the subgroup $\Ds^s_\mu(\R^m) \subset \Ds^s(\R^m)$ consisting of volume preserving diffeomorphisms of $(\R^m,\mu)$, i.e.
\[
	\Ds^s_\mu(\R^m)=\{\varphi \in \Ds^s(\R^m) \;|\; \varphi^\ast \mu=\mu \},
\]
is a closed analytic submanifold of $\Ds^s(\R^m)$.\\
Suppose that $u \in C([0,T];H^s_\omega(\R^{2n};\R^{2n}))$ is a solution to \eqref{symplectic_euler} for some $T > 0$. From \cite{lagrangian} we know that we can integrate $u$ to a unique $\varphi \in C^1([0,T];\Ds^s(\R^{2n}))$ satisfying
\[
	\partial_t \varphi(t)=u(t) \circ \varphi(t),\; 0 \leq t \leq T,\;\varphi(0)=\text{id}.
\]
One has than actually $\varphi \in C^1([0,T];\Ds^s_\omega(\R^{2n}))$ as we will see later on. We will show that one can write the system \eqref{symplectic_euler} in terms of the Lagrangian variable $\varphi$ in the style of Arnold \cite{arnold} who formally worked this out for the Euler equations \eqref{euler}. More precisely

\begin{Th}\label{th_lagrangian_formulation}
	Let $n \geq 1$ and $s > 2n/2+1$. One can write \eqref{symplectic_euler} in Lagrangian coordinates as
\[
	\varphi_{tt}=F(\varphi,\varphi_t),\;\varphi(0)=\text{id},\;\varphi_t(0)=u_0,
\]
where
	\[
		F:\Ds^s(\R^{2n}) \times H^s(\R^{2n};\R^{2n}) \to H^s(\R^{2n};\R^{2n})
	\]
is an analytic map.
\end{Th}

The smooth analog of Theorem \ref{th_lagrangian_formulation} for a closed manifold $M^{2n}$ was shown in \cite{ebin} by Ebin. In \cite{lagrangian} the analog of Theorem \ref{th_lagrangian_formulation} for $\R^m,m \geq 2$, was shown for system \eqref{euler}.\\ \\
Let $u_0 \in H^s_\omega(\R^{2n};\R^{2n})$. We denote by $u(t;u_0)$ the time $t$ value of the solution $u$ to \eqref{symplectic_euler}. From Theorem \ref{th_gwp} we know that $u(t;u_0)$ is defined for all $t \geq 0$ and we know that the time $T$ solution map
\begin{equation}\label{solution_map}
	\Phi_T:H^s_\omega(\R^{2n};\R^{2n}) \to H^s_\omega(\R^{2n};\R^{2n}),\;u_0 \mapsto u(T;u_0),
\end{equation}
is continuous. A natural question is how regular $\Phi_T$ is. We have

\begin{Th}\label{th_nonuniform}
	Let $n \geq 1$, $s > 2n/2+1$ and $T > 0$. Then the time $T$ solution map$\Phi_T$ as in \eqref{solution_map} is nowhere locally uniformly continuous.
\end{Th}

So Theorem \ref{th_nonuniform} tells us that for $T > 0$ the solution map $\Phi_T$ is not more than continuous, that it is nowhere locally Lipschitz and nowhere $C^1$. The analog of Theorem \ref{th_nonuniform} in the case of \eqref{euler} was shown in \cite{euler}.

\section{Eulerian formulation}\label{section_eulerian}

In this section we will write \eqref{symplectic_euler} as an equation on $H^s(\R^{2n};\R^{2n}), s > 2n/2+1$, in terms of $u$ without any constraints. Starting with $u_0 \in H_\omega^s(\R^{2n};\R^{2n})$ will then guarantee that \eqref{symplectic_euler} is satisfied. But let us first work out some properties of $H_\omega$ and its $L^2$ orthogonal complement $H_\omega^\perp$. Suppose that $X=(X_1,\ldots,X_{2n}) \in H_\omega$, i.e. we have $L_X \omega=0$. Note that by $\mu=\frac{1}{n!} \omega^n$ and the Leibniz property $L_X (\omega \wedge \ldots \wedge \omega)=L_X\omega \wedge \ldots \wedge \omega + \ldots + \omega \wedge \ldots \wedge L_X\omega=0$ this implies that $X \in H_\mu$, i.e. that $X$ is divergence free. From Cartan's magic formula we get
\[
	L_X \omega = (\imath_X \circ d + d \circ \imath_X) \omega = d  \imath_X \omega,
\]
where $d$ is the exterior derivative and $\imath_X$ is the contraction with $X$. In the following we represent a one form $\eta=\sum_{1 \leq k \leq 2n} a_k dx_k$ as a row vector $\eta=(a_1,\ldots,a_{2n}) \in \R^{2n}$ and a two form $\xi=\sum_{1 \leq i < j \leq 2n} b_{ij} dx_i \wedge dx_j$ as a skew symmetric matrix $\xi=(\xi_{ij})_{1 \leq i,j \leq 2n} \in \R^{2n \times 2n}_\text{skew}$ where
\[
	\xi_{ij}=\begin{cases} b_{ij},\;i < j,\\
		0,\;i=j,\\
	-b_{ij},\; i > j,\end{cases}
\]
for $1 \leq i,j \leq 2n$. Thus
\[
	\omega=\left(\begin{matrix}
		0 & 1 &  & & \\
		-1 & 0 &  & & \\
		   &   & \ddots & &\\
		   &   &        & 0 & 1 \\
		   &   &        & -1 & 0 
	\end{matrix}\right).
\]
Now consider the differential operator $P:C_c^\infty(\R^{2n};\R^{2n}) \to C_c^\infty(\R^{2n};\R^{2n \times 2n}_\text{skew})$ given by
\[
P(X)=d \imath_X \omega.
\]
Here $C_c^\infty$ refers to smooth and compactly supported. In matrix form we have for the row vector $X=(X_1,\ldots,X_{2n}) \in C_c^\infty(\R^{2n};\R^{2n})$ 
\[
	\imath_X \omega= X \cdot \omega \text{ resp. } d(X)=D(X^\top)-D(X^\top)^\top,
\]
where $D(X^\top)=(\partial_j X_i)_{1 \leq i,j \leq 2n}$ is the Jacobian matrix of $X^\top$ and $^\top$ is the transposition operator. The operator $P$ is then given by
\begin{equation}\label{operator_p}
	P(X)=\omega^\top \cdot D(X^\top)-D(X^\top)^\top \cdot \omega.
\end{equation}
The formal $L^2$-adjoint $P^\ast:C_c^\infty(\R^{2n};\R^{2n \times 2n}_\text{skew}) \to C_c^\infty(\R^{2n})$ is given by
\[
	P^\ast:Y=(Y_{ij})_{1 \leq i,j \leq 2n} \mapsto  -2 (\sum_{i=1}^{2n} \partial_i Y_{i1},\ldots,\sum_{i=1}^{2n} \partial_i Y_{i\,2n}) \cdot \omega.
\]
Note that we have $\langle P^\ast(Y),X \rangle_{L^2}=\langle Y,P(X) \rangle_{L^2}$ for $Y \in H^1(\R^{2n};\R^{2n \times 2n}_\text{skew})$ and $X \in H^1(\R^{2n};\R^{2n})$. In particular we have $P^\ast(Y) \in H_\omega^\perp$. If we denote by 
\[
	\operatorname{div}:Y=(Y_{ij})_{1 \leq i,j \leq 2n} \mapsto (\sum_{i=1}^{2n} \partial_i Y_{i1},\ldots,\sum_{i=1}^{2n} \partial_i Y_{i\,2n})
\]
the divergence operator we can write $P^\ast(Y)=-2 \operatorname{div}(Y) \cdot \omega$. The reason that we are considering $P^\ast(Y)$ is the same as in the case of \eqref{euler} where we consider $-\nabla p$. The $L^2$ orthogonal complement $H_\omega^\perp$ is by the Hodge decomposition the space of vector fields of the form $P^\ast(Y)$. We won't need this here. We will just use this assumption to get the right expressions. So we write the first equation of \eqref{symplectic_euler} as
\[
	u_t+(u \cdot \nabla)=P^\ast(Y)
\]
for some unknown $Y$. Let us define $\Omega(Y)=d(\operatorname{div}Y)$, i.e.  
\[
	\Omega_{k\ell}(Y)=\partial_\ell (\operatorname{div}Y)_k-\partial_k (\operatorname{div}Y)_\ell,\; 1 \leq k,\ell \leq 2n.
\]
With this we have $PP^\ast(Y)=2\Omega(Y)$. Taking the divergence of $\Omega(Y)$ leads for $\ell=1,\ldots,2n$ to
\[
	(\operatorname{div}\Omega(Y))_\ell=\sum_{k=1}^{2n} \partial_k \Omega_{k\ell}(Y) = \sum_{k=1}^{2n} \partial_k \partial_\ell (\operatorname{div}Y)_k - \Delta (\operatorname{div}Y)_\ell=- \Delta (\operatorname{div}Y)_\ell,
\]
since
\[
	\sum_{k=1}^{2n} \partial_k \partial_\ell (\operatorname{div}Y)_k=\partial_\ell \sum_{i,k=1}^{2n} \partial_k \partial_i Y_{ik}=0,\; \ell=1,\ldots,2n,
\]
do to the antisymmetry of $Y$. Here we use $\Delta=\sum_{1 \leq \ell \leq 2n} \partial_\ell^2$ to denote the Laplacian. So we have
\[
	-\Delta \operatorname{div} Y = \operatorname{div} \Omega(Y).
\]
By formally inverting the Laplacian we can write
\begin{equation}\label{delta_formula}
	P^\ast(Y)=2 \Delta^{-1} \operatorname{div} \Omega(Y) \cdot \omega= \Delta^{-1} \operatorname{div} P P^\ast(Y) \cdot \omega=-\frac{1}{2} \Delta^{-1} P^\ast P P^\ast(Y).
\end{equation}
Let us apply these considerations to \eqref{symplectic_euler}. If we apply $P$ to the equation $u_t+(u \cdot \nabla) u = P^\ast(Y)$ we get due to the assumption $u \in H_\omega$
\[
	P((u \cdot \nabla)u)=P P^\ast(Y).
\]
To calculate $P((u \cdot \nabla)u)$ we apply the product rule. The terms where $P$ hits the second $u$ vanish because $P(u)=0$. Thus we end up with
\begin{align*}
	&P((u \cdot \nabla)u)=\omega^\top \cdot \left(\begin{matrix}\sum_{k=1}^{2n} \partial_1 u_k \partial_k u_1 & \cdots & \sum_{k=1}^{2n} \partial_{2n} u_k \partial_k u_1 \\ \vdots & \ddots & \vdots\\ \sum_{k=1}^{2n} \partial_1 u_k \partial_k u_{2n} & \cdots & \sum_{k=1}^{2n} \partial_{2n} u_k \partial_k u_{2n} \end{matrix}\right)\\
		& -\left(\begin{matrix}\sum_{k=1}^{2n} \partial_1 u_k \partial_k u_1 & \cdots & \sum_{k=1}^{2n} \partial_{2n} u_k \partial_k u_1 \\ \vdots & \ddots & \vdots\\ \sum_{k=1}^{2n} \partial_1 u_k \partial_k u_{2n} & \cdots & \sum_{k=1}^{2n} \partial_{2n} u_k \partial_k u_{2n} \end{matrix}\right)^\top \cdot \omega=:P_H(u).
\end{align*}
Since $u \in H_\omega$ implies that $u$ is divergence free we can write $(\partial_\ell u \cdot \nabla) u=\sum_{k=1}^{2n} \partial_k (\partial_\ell u_k u),\ell=1,\ldots,2n$. Using this representation we have
\begin{align*}
	&P((u \cdot \nabla)u)=\omega^\top \cdot \left(\begin{matrix}\sum_{k=1}^{2n} \partial_k (\partial_1 u_k u_1) & \cdots & \sum_{k=1}^{2n} \partial_k (\partial_{2n} u_k u_1) \\ \vdots & \ddots & \vdots\\ \sum_{k=1}^{2n} \partial_k (\partial_1 u_k u_{2n}) & \cdots & \sum_{k=1}^{2n} \partial_k(\partial_{2n} u_k u_{2n}) \end{matrix}\right)\\
		&- \left(\begin{matrix}\sum_{k=1}^{2n} \partial_k (\partial_1 u_k u_1) & \cdots & \sum_{k=1}^{2n} \partial_k (\partial_{2n} u_k u_1) \\ \vdots & \ddots & \vdots\\ \sum_{k=1}^{2n} \partial_k (\partial_1 u_k u_{2n}) & \cdots & \sum_{k=1}^{2n} \partial_k(\partial_{2n} u_k u_{2n}) \end{matrix}\right)^\top \cdot \omega=:P_L(u).
\end{align*}
We will use the above two representations for $P((u \cdot \nabla)u)=P_H(u)=P_L(u)$ to invert the Laplacian in a rigorous way. We will use $P_H$ for the higher Fourier modes and $P_L$ for the lower Fourier modes. This trick to invert the Laplacian was used in \cite{chemin} by Chemin for the Euler equations \eqref{euler} (see also \cite{lagrangian} where this trick was used for \eqref{euler} too). To do the splitting in Fourier space consider the following Fourier cut-off multiplier on the Sobolev space $H^\sigma(\R^{2n}),\sigma \geq 0$,
\[
	\chi(D):H^\sigma(\R^{2n}) \to H^\sigma(\R^{2n}),f \mapsto \mathcal F^{-1}[\chi(\xi) \hat f(\xi)],
\]
where $\mathcal F$ is the Fourier Transform operator and $\chi(\xi),\xi \in \R^{2n}$, is the indicator function of the closed unit ball in $\R^{2n}$, i.e.
\[
	\chi(\xi)=\begin{cases} 1,&\;|\xi| \leq 1,\\0,&\;|\xi| > 1.\end{cases}
\]
Note that $\chi(D)$ is an infinitely smoothing bounded operator. The inverse of the Laplacian $\Delta^{-1}$ is a Fourier multiplier operator with symbol $-1/|\xi|^2$. For $s > 2n/2+1$ and $1 \leq i,j,k \leq 2n$ 
\[
	H^s(\R^{2n}) \times H^s(\R^{2n}) \to H^s(\R^{2n}),\;(f,g) \mapsto \Delta^{-1} (1-\chi(D)) \partial_i (\partial_j f \cdot \partial_k g) 
\]
is a bounded bilinear map since $H^{s-1}(\R^{2n})$ is a Banach algebra and the map $\Delta^{-1} (1-\chi(D)):H^{s-2}(\R^{2n}) \to H^s(\R^{2n})$ is bounded and
\[
	H^s(\R^{2n}) \times H^s(\R^{2n}) \to H^s(\R^{2n}),\;(f,g) \mapsto \Delta^{-1} \chi(D) \partial_i \partial_j (\partial_k f \cdot g) 
\]
is a bounded bilinear map as well since $\partial_i \partial_j \Delta^{-1} \chi(D):H^{s-1}(\R^{2n}) \to H^s(\R^{2n})$ is bounded. So
\[
	H^s(\R^{2n};\R^{2n}) \to H^s(\R^{2n};\R^{2n}), u \mapsto \Delta^{-1} (1-\chi(D)) P^\ast P_H(u)
\]
resp.
\[
	H^s(\R^{2n};\R^{2n}) \to H^s(\R^{2n};\R^{2n}), u \mapsto \Delta^{-1} \chi(D) P^\ast P_L(u)
\]
are bounded quadratic forms. Now we write
\[
	P((u \cdot \nabla)u)=(1-\chi(D)) P_H(u)+\chi(D) P_L(u)=PP^\ast(Y).
\]
Motivated by \eqref{delta_formula} we apply $P^\ast$ and $\Delta^{-1}$ to get
\[
	P^\ast(Y)=-\frac{1}{2} \Delta^{-1} (1-\chi(D)) P^\ast P_H(u)  -\frac{1}{2} \Delta^{-1} \chi(D) P^\ast P_L(u) 
\]
So we can replace \eqref{symplectic_euler} by
\begin{equation}\label{symplectic_euler_new}
	u_t + (u \cdot \nabla) u=B(u),\;u(0)=u_0 \in H^s_\omega(\R^{2n};\R^{2n}),
\end{equation}
where $B$ is the bounded quadratic form
\begin{align}
	\label{b_formula}
	\begin{split}
	&B:H^s(\R^{2n};\R^{2n}) \to H^s(\R^{2n};\R^{2n}),\\
	&u  \mapsto -\frac{1}{2} \Delta^{-1} (1-\chi(D)) P^\ast P_H(u) -\frac{1}{2} \Delta^{-1} \chi(D) P^\ast P_L(u).
	\end{split}
\end{align}
To see that $B(u) \in H^s_\omega(\R^{2n};\R^{2n})^\perp$ we take $w \in H^s_\omega(\R^{2n};\R^{2n})$, i.e. $P(w)=0$ and consider $\langle B(u),w \rangle_{L^2}$. For the $P_H$ part of $B(u)$ we have
\[
	\langle -\frac{1}{2} \Delta^{-1} (1-\chi(D)) P^\ast P_H(u),w \rangle_{L^2}=
	\langle -\frac{1}{2} \Delta^{-1} (1-\chi(D)) P_H(u),P(w) \rangle_{L^2}=0.
\]
For the $P_L$ part of $B(u)$ we have
\[
	\langle -\frac{1}{2} \Delta^{-1} \chi(D) P^\ast P_L(u),w \rangle_{L^2} = \langle \frac{1}{2} (-\Delta)^{-1/2} \chi(D) P_L(u),(-\Delta)^{-1/2} P(w) \rangle_{L^2}=0.
\]
Thus we have $\langle B(u),w \rangle_{L^2}=0$ and hence $B(u) \in H^s_\omega(\R^{2n};\R^{2n})^\perp$.
In the following we will show that \eqref{symplectic_euler} and \eqref{symplectic_euler_new} are equivalent. 

\begin{Lemma}\label{lemma_old_to_new}
	Let $s > 2n/2+1$. Suppose that $u \in C([0,T];H^s_\omega(\R^{2n};\R^{2n}))$ is a solution to \eqref{symplectic_euler} for some $T > 0$. Then $u$ is a solution to \eqref{symplectic_euler_new} on $[0,T]$.
\end{Lemma}
\begin{proof}
	Suppose that $u \in C([0,T];H^s_\omega(\R^{2n};\R^{2n}))$ is a solution to \eqref{symplectic_euler}. We define
	\[
		Z=u_t+(u \cdot \nabla) u \in C([0,T];H^{s-1}_\omega(\R^{2n};\R^{2n})^\perp).
	\]
By using \eqref{delta_formula} we get
	\[
		P^\ast P(B(u))=P^\ast((1-\chi(D)) P_H(u) + \chi(D) P_L(u))=P^\ast P((u \cdot \nabla) u))=P^\ast P(Z)
	\]
	since $u \in H^s_\omega(\R^{2n};\R^{2n})$. So we have $P(B(u)-Z)=0$, i.e. $B(u)-Z \in H^{s-1}_\omega(\R^{2n};\R^{2n})$. On the other hand we have $B(u)-Z \in H^{s-1}_\omega(\R^{2n};\R^{2n})^\perp$. So we conclude $Z=B(u)$. This shows that $u$ is a solution to \eqref{symplectic_euler_new}. Note that the proof shows in particular that $u_t + (u \cdot \nabla)u$ is in $H^s$.
\end{proof}

The other direction reads as

\begin{Lemma}\label{lemma_new_to_old}
	Let $s > 2n/2+1$. Suppose that $u \in C([0,T];H^s(\R^{2n};\R^{2n}))$ is a solution to \eqref{symplectic_euler_new} for some $T > 0$. Then $u \in C([0,T];H^s_\omega(\R^{2n};\R^{2n}))$. In particular $u$ is a solution to \eqref{symplectic_euler} on $[0,T]$.
\end{Lemma}

\begin{proof}
	Taking the $t$ derivative of $\frac{1}{2} \langle P(u),P(u) \rangle_{L^2}$ gives
	\begin{align*}
		\frac{1}{2} \frac{d}{dt} \langle P(u),P(u) \rangle_{L^2}=\langle PB(u),P(u) \rangle_{L^2}-\langle P((u \cdot \nabla)u),P(u) \rangle_{L^2}. 	
	\end{align*}
We have by \eqref{delta_formula}
\begin{align*}
	\langle PB(u),P(u) \rangle_{L^2}&=\langle P^\ast PB(u),u \rangle_{L^2}\\&=\langle (1-\chi(D)) P_H(u) +\chi(D) P_L(u),P(u) \rangle_{L^2}.
\end{align*}
	If we carry out the derivative in $P_L(u,u)$ we get
\[
	P_L(u)=P_H(u) + Q(u),
\]
where
	\begin{align*}
		Q(u)&=\omega^\top \cdot \left(\begin{matrix}u_1 \partial_1 \operatorname{div}u & \cdots &  u_1 \partial_1 \operatorname{div}u \\ \vdots & \ddots & \vdots\\ u_{2n} \partial_1 \operatorname{div}u  & \cdots & u_{2n} \partial_{2n} \operatorname{div}u\end{matrix}\right)\\
			&- \left(\begin{matrix}u_1 \partial_1 \operatorname{div}u & \cdots &  u_1 \partial_1 \operatorname{div}u \\ \vdots & \ddots & \vdots\\ u_{2n} \partial_1 \operatorname{div}u  & \cdots & u_{2n} \partial_{2n} \operatorname{div}u\end{matrix}\right)^\top \cdot \omega
	\end{align*}
Thus 
	\begin{align*}
		&(1-\chi(D)) P_H(u) +\chi(D) P_L(u) =P_H(u) + \chi(D) Q(u).
	\end{align*}
	For $P((u \cdot \nabla)u)$ we have
	\[
		P((u \cdot \nabla) u)=P_H(u)+ (u \cdot \nabla) P(u).	
	\]
Combining the above expressions gives
	\[
		\frac{1}{2} \frac{d}{dt} \langle P(u),P(u) \rangle_{L^2}=\langle \chi(D)Q(u),P(u) \rangle_{L^2}-\langle (u \cdot \nabla) P(u),P(u) \rangle_{L^2}.
	\]
Note that $P(u)$ has components $\partial_{2i-1}u_{2i-1}+\partial_{2i} u_{2i},i=1,\ldots,n$. Thus we have $\|\operatorname{div}u\|_{L^2} \leq C \|P(u)\|_{L^2}$ for some $C > 0$ independent of $u$. Now we have for functions $f,g,h$ and $i=1,\ldots,2n$
	\[
		\langle \chi(D)(f \partial_i g),h \rangle_{L^2}=-\langle g, \partial_i (f \chi(D) h) \rangle_{L^2}.
	\]
Applying Cauchy-Schwarz gives
	\begin{align}
		\label{Q_estimate}
		\begin{split}
		|\langle \chi(D)(f \partial_i g),h \rangle_{L^2}| &=|\langle g, \partial_i (f \chi(D) h) \rangle_{L^2}| \leq \|g\|_{L^2} \|f \chi(D)h\|_{H^1} \\
		&\leq C \|g\|_{L^2} \|f\|_{H^s}  \|\chi(D)h\|_{H^1} \leq C \|g\|_{L^2} \|f\|_{H^s} \|h\|_{L^2},
		\end{split}
	\end{align}
	for some $C > 0$. Here we used that multiplication $H^s \times H^1 \to H^1$ is bounded for $s > 2n/2+1$ and that the smoothing operator $\chi(D):L^2 \to H^1$ is bounded. So we get by using $f=u,g=\operatorname{div}u,h=P(u)$ above
	\[
		|\langle \chi(D)Q(u),P(u) \rangle_{L^2}| \leq C \|u\|_{H^s} \|P(u)\|_{L^2}^2 
	\]
	for some $C > 0$. For $\langle (u \cdot \nabla) P(u),P(u) \rangle_{L^2}$ we have
\[
	\langle (u \cdot \nabla) P(u),P(u) \rangle_{L^2}=-\langle P(u),(u \cdot \nabla) P(u) \rangle_{L^2}-\langle P(u),P(u) \operatorname{div}u\rangle_{L^2}. 
\]
Thus we get
	\begin{align*}
		|\langle (u \cdot \nabla) P(u),P(u) \rangle_{L^2}|&=\frac{1}{2} |\langle P(u),P(u) \operatorname{div}u\rangle_{L^2}| \\
		&\leq \|\operatorname{div}u\|_{L^\infty} \|P(u)\|_{L^2}^2 \leq C \|u\|_{H^s} \|P(u)\|_{L^2}^2,
	\end{align*}
where we used the Sobolev Imbedding $H^s \hookrightarrow C^1_0$. Finally we end up with
	\begin{equation}\label{p_ineq}
		\frac{d}{dt} \|P(u)\|_{L^2}^2 \leq C \|u\|_{H^s} \|P(u)\|_{L^2}^2 
	\end{equation}
	on $[0,T]$ for some $C > 0$ independent of $u$. Using Gr\"onwall's inequality we conclude that if $u_0 \in H_\omega$, i.e. $P(u_0)=0$, then $P(u)=0$ on $[0,T]$. In other words $u \in C([0,T];H^s_\omega(\R^{2n};\R^{2n}))$. So $u$ solves \eqref{symplectic_euler} on $[0,T]$.
\end{proof}

Lemma \ref{lemma_old_to_new} and Lemma \ref{lemma_new_to_old} show that \eqref{symplectic_euler} and \eqref{symplectic_euler_new} are equivalent. So from now on we will work with \eqref{symplectic_euler_new}.

\section{Lagrangian formulation}\label{section_lagrangian}

The goal of this section is to write \eqref{symplectic_euler_new} in Lagrangian variables. But we will consider it without an initial constraint. More precisely let $n \geq 1$ and $s > 2n/2+1$. We consider
\begin{equation}\label{symplectic_euler_extended}
	u_t + (u \cdot \nabla)u=B(u),\;u(0)=u_0 \in H^s(\R^{2n};\R^{2n}),
\end{equation}
where $B(u)$ is given by \eqref{b_formula}. Suppose now that $u \in C([0,T];H^s(\R^{2n};\R^{2n}))$ is a solution to \eqref{symplectic_euler_extended} on some time interval $[0,T],T \geq 0$. From \cite{lagrangian} we know that $u$ generates a unique flow map $\varphi \in C^1([0,T];\Ds^s(\R^{2n}))$ satisfying
\begin{equation}\label{flow_map}
	\varphi_t(t)=u(t) \circ \varphi(t), 0 \leq t \leq T,\;\varphi(0)=\text{id}.
\end{equation}
Here $\text{id}:\R^{2n} \to \R^{2n}$ is the identity diffeomorphism. Taking the $t$ derivative of $\varphi_t = u \circ \varphi$ gives
\[
	\varphi_{tt}=(u_t + (u \cdot \nabla)) \circ \varphi=B(u) \circ \varphi=B(\varphi_t \circ \varphi^{-1}) \circ \varphi.
\]
This computation motivates us to define $F(\varphi,v)=B(v \circ \varphi^{-1}) \circ \varphi$. By the regularity properties proved in \cite{composition} we know that for $s > 2n/2+1$ and $0 \leq s' \leq s$
\[
	H^{s'}(\R^{2n}) \times \Ds^s(\R^{2n}) \to H^{s'}(\R^{2n}),\;(f,\varphi) \mapsto f \circ \varphi,
\]
and
\[
	\Ds^s(\R^{2n}) \to \Ds^s(\R^{2n}),\; \varphi \mapsto \varphi^{-1},
\]
are continuous. Thus we see by the boundedness of $B$ in $H^s$ that
\begin{align}
	\label{f_map}
	\begin{split}
		F:\Ds^s(\R^{2n}) \times H^s(\R^{2n};\R^{2n}) &\to H^s(\R^{2n};\R^{2n}),\\
	(\varphi,v) &\mapsto F(\varphi,v)=B(v \circ \varphi^{-1}) \circ \varphi,
	\end{split}
\end{align}
is continuous. We actually have more regularity.

\begin{Lemma}\label{f_regularity}
	Let $n \geq 1$ and $s > 2n/2+1$. Then the map $F$ in \eqref{f_map} is real analytic.
\end{Lemma}

For basic concepts of real analyticity in Hilbert spaces see \cite{lagrangian}.

\begin{proof}[Proof of Lemma \ref{f_regularity}]
	From \cite{lagrangian} we know that for $1 \leq s' \leq s$ and $j=1,\ldots,2n$
	\[
		\Ds^s(\R^{2n}) \times H^{s'}(\R^{2n}) \to H^{s'-1}(\R^{2n}),\;(\varphi,f) \mapsto [\partial_j(f \circ \varphi^{-1})] \circ \varphi,
	\]
	is analytic. The reason is essentially that after expanding $[\partial_j(f \circ \varphi^{-1})] \circ \varphi$ we get rid of the composition. Again in \cite{lagrangian} it was shown that 
\[
	\Ds^s(\R^{2n}) \times H^{s-2}(\R^{2n}) \to H^s(\R^{2n}),\;(\varphi,f) \mapsto \left(\Delta^{-1}(1-\chi(D))(f \circ \varphi^{-1})\right) \circ \varphi,
\]
	is analytic. The expression $P^\ast P_H(u)$ is a linear combination of expressions of the form $\partial_i (\partial_j u_p \cdot \partial_k u_q), 1 \leq i,j,k,p,q \leq 2n$. The map
	\begin{align*}
		\D^s(\R^{2n}) \times H^s(\R^{2n}) \times H^s(\R^{2n}) &\to H^{s-2}(\R^{2n}),\\
		(\varphi,f,g) &\mapsto (\partial_i (\partial_j (f \circ \varphi^{-1}) \cdot \partial_k (g \circ \varphi^{-1}))) \circ \varphi
	\end{align*}
is analytic. One can see this by writing
	\begin{align*}
		&(\partial_i (\partial_j (f \circ \varphi^{-1}) \cdot \partial_k (g \circ \varphi^{-1}))) \circ \varphi\\
		&=\partial_i [(\partial_j (f \circ \varphi^{-1}) \circ \varphi \cdot \partial_k (g \circ \varphi^{-1}) \circ \varphi) \circ \varphi^{-1}] \circ \varphi.
	\end{align*}
Thus we get that
	\begin{align*}
		&\Ds^s(\R^{2n}) \times H^s(\R^{2n};\R^{2n}) \to H^s(\R^{2n};\R^{2n}),\\
		&(\varphi,v) \mapsto \left(\Delta^{-1}(1-\chi(D)) P^\ast P_H(v \circ \varphi^{-1})\right) \circ \varphi,
	\end{align*}
	is analytic. The term $\Delta^{-1}\chi(D)P^\ast P_L(u)$ is a linear combination of expressions of the form $\partial_i \partial_j\chi(D) (\partial_k u_p \cdot u_q),1\leq i,j,k,p,q \leq 2n$. In \cite{lagrangian} it was shown that
\[
	\Ds^s(\R^{2n}) \times H^{s-1}(\R^{2n}) \to H^s(\R^{2n}),\;(\varphi,f) \mapsto (\partial_i \partial_j \Delta^{-1} \chi(D) (f \circ \varphi^{-1})) \circ \varphi,
\]
is analytic. Thus we see that	
\begin{align*}
		&\Ds^s(\R^{2n}) \times H^s(\R^{2n};\R^{2n}) \to H^s(\R^{2n};\R^{2n}),\\
		&(\varphi,v) \mapsto \left(\Delta^{-1}\chi(D) P^\ast P_L(v \circ \varphi^{-1})\right) \circ \varphi,
\end{align*}
is analytic. Putting the two parts together we see that $F$ is analytic.
\end{proof}

With this additional regularity for $F$ we can consider the following analytic second order ODE on $\Ds^s(\R^{2n})$
\begin{equation}\label{second_order_ode}
	\varphi_{tt}=F(\varphi,\varphi_t),\;\varphi(0)=\text{id},\;\varphi_t(0)=u_0 \in H^s(\R^{2n};\R^{2n}).
\end{equation}
One can convert \eqref{second_order_ode} to a system of first order ODEs on the tangent bundle $T\Ds^s(\R^{2n})=\Ds^s(\R^{2n}) \times H^s(\R^{2n};\R^{2n})$
\begin{equation}\label{first_order_ode}
	\frac{d}{dt} \left(\begin{array}{c} \varphi \\ v \end{array}\right)=\left(\begin{array}{c} v \\ F(\varphi,v) \end{array}\right),\;\left(\begin{array}{c} \varphi(0)\\ v(0) \end{array}\right)=\left(\begin{array}{c} \text{id} \\ u_0 \end{array}\right)
\end{equation}
Note that $F(\varphi,v)$ is quadratic in $v$. In Differential Geometry people call the vector field on $T\Ds^s(\R^{2n})$ appearing on the right hand side of \eqref{first_order_ode} a spray and the equation \eqref{second_order_ode} a geodesic equation -- see \cite{ebin_marsden,ebin}. Due to the regularity of $F$ we can apply the Picard-Lindel\"of Theorem to get local solutions to \eqref{first_order_ode} and with that local solutions to \eqref{symplectic_euler_extended}. More precisely

\begin{Lemma}\label{local_solutions}
	Let $n \geq 1$, $s > 2n/2+1$ and $T > 0$. Suppose $(\varphi,v)$ solves \eqref{first_order_ode} on $[0,T]$. Then $u(t):=v(t) \circ \varphi(t)^{-1},0 \leq t \leq T$, solves \eqref{symplectic_euler_extended}. On the other hand let $u \in C([0,T];H^s(\R^{2n};\R^{2n}))$ be a solution to \eqref{symplectic_euler_extended} and $\varphi \in C^1([0,T];\Ds^s(\R^{2n}))$ its flow map as in \eqref{flow_map}. Then $(\varphi,v):=(\varphi,\varphi_t)$ is a solution to \eqref{first_order_ode} on $[0,T]$.
\end{Lemma}

\begin{proof}
	Suppose that $(\varphi,v) \in C^1([0,T];\Ds^s(\R^{2n}) \times H^s(\R^{2n};\R^{2n}))$ is a solution to \eqref{first_order_ode}. We define $u(t)=v(t) \circ \varphi(t)^{-1},0 \leq t \leq T$. From the regularity of the composition we know that $u \in C([0,T];H^s(\R^{2n};\R^{2n})$. As a composition of $C^1$ maps we have that $u \in C^1([0,T] \times \R^{2n};\R^{2n})$. Thus we get pointwise for $x \in \R^{2n}$ by taking the $t$ derivative in $\varphi_t=u \circ \varphi$
\[
	\varphi_{tt}(t)=(u_t(t)+(u(t) \cdot \nabla)u(t)) \circ \varphi(t) = F(\varphi(t),v(t)),\; 0 \leq t \leq T.
\]
So we have pointwise
	\begin{align*}
		u_t(t) &= F(\varphi(t),v(t)) \circ \varphi(t)^{-1}-(u(t) \cdot \nabla) u(t)\\
		&=B(u(t))-(u(t) \cdot \nabla) u(t),\; 0 \leq t \leq T.
	\end{align*}
Integrating gives pointwise
	\[
		u(t)=\int_0^t \left(B(u(s))-(u(s) \cdot \nabla)u(s)\right) \;ds,\;0 \leq t \leq T.
	\]
Since the integrand is in $C([0,T];H^{s-1}(\R^{2n};\R^{2n}))$ the above is an identity in $H^{s-1}$. Hence $u$ is a solution to \eqref{symplectic_euler_extended}.\\
	Suppose now that $u \in C([0,T];H^s(\R^{2n};\R^{2n}))$ is a solution to \eqref{symplectic_euler_extended} and let $\varphi \in C^1([0,T];\Ds^s(\R^{2n}))$ be its flow map as in \eqref{flow_map}. Note that $u \in C^1([0,T] \times \R^{2n};\R^{2n})$. Taking pointwise the $t$ derivative in $\varphi_t(t)=u(t) \circ \varphi(t)$ gives
	\begin{align*}
		\varphi_{tt}(t)&=(u_t(t) + (u(t) \cdot \nabla)u(t)) \circ \varphi(t)\\
		&=B(u(t)) \circ \varphi(t)=F(\varphi(t),\varphi_t(t)),\;0 \leq t \leq T.
	\end{align*}
Integrating gives pointwise
	\[
		\varphi_t(t)=\int_0^t F(\varphi(s),\varphi_t(s)) \;ds,0 \leq t \leq T.
	\]
Since the integrand is in $C([0,T];H^s(\R^{2n};\R^{2n}))$ we get
	\[
		\varphi_{tt}(t)=F(\varphi(t),\varphi_t(t)),\;0 \leq t \leq T,
	\]
as an identity in $H^s$. Hence $(\varphi,\varphi_t)$ solves \eqref{first_order_ode} on $[0,T]$. This finishes the proof.
\end{proof}

As a corollary of Lemma \ref{local_solutions} we get the local well-posedness of \eqref{symplectic_euler_extended}.

\begin{Coro}\label{coro_lwp}
Let $n \geq 1$ and $s > 2n/2+1$. Then \eqref{symplectic_euler_extended} is locally well-posed in the sense of Hadamard.
\end{Coro}

\begin{proof}
	Existence and Uniqueness of solutions to \eqref{symplectic_euler_extended} follow from Lemma \ref{local_solutions} and the corresponding statement for ODEs applied to \eqref{first_order_ode}. Continuous dependence on the initial value $u_0 \in H^s(\R^{2n};\R^{2n}) \mapsto u(T)=v(T) \circ \varphi(T)^{-1}$ follows from the corresponding statement for ODEs and the continuity of the composition.
\end{proof}

As a consequence we get that \eqref{symplectic_euler} is locally well-posed in $H^s(\R^{2n})$ for $n \geq 1$ and $s > 2n/2+1$. Another consequence is that \eqref{second_order_ode} is a Lagrangian formulation of \eqref{symplectic_euler}. In other words we get a proof of Theorem \ref{th_lagrangian_formulation}. 

\begin{proof}[Proof of \ref{th_lagrangian_formulation}]
	The proof follows from Lemma \ref{lemma_old_to_new}, Lemma \ref{lemma_new_to_old}, Lemma \ref{local_solutions} and Corollary \ref{coro_lwp}.
\end{proof}

\section{Global well-posedness} \label{section_gwp}

Throughout this section we assume $n \geq 1$ and $s > 2n/2+1$. The goal of this section is to prove that \eqref{symplectic_euler} is globally well-posed in $H^s(\R^{2n})$. Before doing that we will do some prelimary considerations. We denote by $(\varphi(t;u_0),v(t;u_0))$ the solution to \eqref{first_order_ode} with initial value $(\text{id},u_0)$ on its maximal time of existence $0 \leq t < T^\ast(u_0)$. Here we denote by $T^\ast(u_0) \in (0,\infty]$ the life span of the solution to \eqref{first_order_ode} starting at $(\text{id},u_0)$. Note that due to the quadratic nature of $F(\varphi,v)$ in $v$ the system \eqref{first_order_ode} enjoys the following scaling: Let $(\varphi,v)$ be a solution to \eqref{first_order_ode} on $[0,T]$ for some $T > 0$. For $\lambda > 0$ consider the scaled quantities
\begin{equation}\label{scaling}
	\varphi^\lambda(t)=\varphi(\lambda t),\;v^\lambda(t)=\lambda v(\lambda t),\;0 \leq t \leq T/\lambda.
\end{equation}
Then $(\varphi^\lambda,v^\lambda)$ solves \eqref{first_order_ode} on $[0,T/\lambda]$ with initial values $(\varphi^\lambda(0),v^\lambda(0))=(\text{id},\lambda u_0)$. In other words we have $T^\ast(\lambda u_0)=\frac{1}{\lambda} T^\ast(u_0)$ and $\varphi(t;\lambda u_0)=\varphi(\lambda t;u_0)$ for $0 \leq t < \frac{1}{\lambda} T^\ast(u_0)$.\\
We define $U \subset H^s(\R^{2n};\R^{2n})$ to be the set of initial values $u_0$ for which \eqref{first_order_ode} has a solution beyond time $T=1$. Note that $U$ is an open neighborhood of $0 \in H^s(\R^{2n};\R^{2n})$ and star shaped w.r.t. $0$. We prove the following blow up criterion.

\begin{Lemma}\label{blowup_hs}
Let $u_0 \in H^s_\omega(\R^{2n};\R^{2n})$ and $u$ the solution to \eqref{symplectic_euler} on its maximal time of existence $[0,T^\ast(u_0))$. If $T^\ast(u_0) < \infty$ then $\lim_{t \uparrow T^\ast(u_0)} \|u(t)\|_{H^s}=\infty$.
\end{Lemma}

\begin{proof}
Suppose not. Then there is $M > 0$ and $(t_k)_{k \geq 1}$ with $t_k \uparrow T^\ast(u_0)$ s.t.
	\[
		\|u(t_k)\|_{H^s} \leq M,\;k \geq 1.
	\]
	There is $\delta > 0$ s.t $B_{2\delta}(0) \subset U$, where $U \subset H^s(\R^{2n};\R^{2n})$ is the set of initial values for which the solution to \eqref{symplectic_euler} exists beyond time $T=1$. This means that the solution to \eqref{symplectic_euler} with initial value $w_0 \in H^s_\omega(\R^{2n};\R^{2n})$ satisfying $\|w_0\|_{H^s} \leq \delta$ exists beyond time $T=1$. Thus by scaling we see that the solution to \eqref{symplectic_euler} with an initial value $w_0$ satisfying $\|w_0\|_{H^s} \leq M$ exists beyond time $T=\delta/M$. Now take $k \geq 1$ s.t. $T^\ast(u_0)-t_k < \delta/M$. We can continue the solution from $u(t_k)$ to the interval $[t_k,t_k+\delta/M]$ which is beyond $T^\ast(u_0)$. This is a contradiction. So we have $\lim_{t \uparrow T^\ast(u_0)} \|u(t)\|_{H^s}=\infty$.
\end{proof}

We will refine the above blow up criterion \`a la Beale-Kato-Majda \cite{bkm}. But before doing that we will introduce some notation. For $N > 0$ we define $\chi_N(D)$ to be the Fourier multiplier operator with symbol
\[
	\chi_N(\xi)=\begin{cases} 1,\;|\xi| \leq N,\\0,\;|\xi| > N,\end{cases}
\]
		where $\xi \in \R^{2n}$, i.e. $\chi_N(\xi)$ is the characteristic function of the ball of radius $N > 0$ around $0$ in $\R^{2n}$. We have for $s' \geq 0$
\[
	\chi_N(D):H^{s'}(\R^{2n}) \to H^\infty(\R^{2n})=\cap_{\rho \geq 0} H^\rho(\R^{2n}),
\]
		and $\chi_N(D)f \to f$ in $H^{s'}$ as $N \to \infty$ for $f \in H^{s'}(\R^{2n})$. For $s' \geq 0$ we denote by $J^{s'}$ the Fourier multiplier operator with symbol $(1+|\xi|^2)^{s'/2},\xi \in \R^{2n}$. In particular we have $\|f\|_{H^{s'}}=\|J^{s'}(f)\|_{L^2}$ for $f \in H^{s'}(\R^{2n})$.

\begin{Lemma}\label{blowup_sup}
	Let $u_0 \in H^s_\omega(\R^{2n};\R^{2n})$ and $u$ the solution to \eqref{symplectic_euler} on its maximal time of existence $[0,T^\ast(u_0))$. If $T^\ast(u_0) < \infty$ then $\int_0^{T^\ast(u_0)} \|\nabla u(t)\|_{L^\infty} \;dt=\infty$. Here $\nabla u$ denotes the collection of first order derivatives of $u$, i.e. the Jacobian matrix of $u$.
\end{Lemma}

\begin{proof}
	We define $u_{0,N}=\chi_N(D)u_0,N \geq 1$ and the corresponding solution to \eqref{symplectic_euler} on its maximal interval of existence $[0,T^\ast(u_{0,N}))$ by $u_N$. Note that $u_N \in H^\infty_\omega(\R^{2n};\R^{2n})=\cap_{\rho \geq 1} H^\rho_\omega(\R^{2n};\R^{2n})$ since $P$ commutes with $\chi_N(D)$. Consider
	\begin{align*}
		\frac{1}{2} \frac{d}{dt} \langle J^s(u_N),J^s(u_N) \rangle_{L^2} =& \langle J^s(B(u_N)),J^s(u_N) \rangle_{L^2}\\
		&-\langle J^s((u_N \cdot \nabla) u_N ),J^s(u_N) \rangle_{L^2}.
	\end{align*}
	We treat the expressions on the right separately. We have
	\[
		\langle J^s(B(u_N)),J^s(u_N)\rangle_{L^2}=\langle B(u_N),J^{2s}(u_N) \rangle_{L^2}=0.
	\]
The reason is that $P,J^{2s}$ commute and that $u_N \in H_\omega,B(u_N) \in H_\omega^\perp$ holds. For the other expression we have
	\begin{align*}
		\langle J^s((u_N \cdot \nabla) u_N ),J^s(u_N) \rangle_{L^2}=\langle J^s((u_N \cdot \nabla) u_N )-(u_N \cdot \nabla)J^s(u_N),J^s(u_N) \rangle_{L^2},
	\end{align*}
	since $\langle (u_N \cdot \nabla) J^s(u_N),J^s(u_N) \rangle_{L^2}=0$ due to $\operatorname{div} u_N=0$. Using the Kato-Ponce commutator estimate in \cite{kato_ponce} we get
\[
	\|J^s((u_N \cdot \nabla) u_N )-(u_N \cdot \nabla)J^s(u_N)\|_{L^2} \leq C \|\nabla u_N\|_{L^\infty} \|u_N\|_{H^s}.
\]
Thus
	\[
		\frac{1}{2} \frac{d}{dt} \|u_N(t)\|_{H^s}^2 \leq C \|\nabla u_N(t)\|_{L^\infty} \|u_N(t)\|_{H^s}^2,\;0 \leq t < T^\ast(u_{0,N}).
\]
So Gr\"onwall's inequality gives
	\[
		\|u_N(T)\|_{H^s} \leq \|u_{0,N}\|_{H^s} e^{C \int_0^T \|\nabla u_N(t)\|_{L^\infty} \;dt},\; 0 \leq T < T^\ast(u_{0,N}).
	\]
Now let $T < T^\ast(u_0)$. For $N \geq 1$ large enough we have $T^\ast(u_{0,N}) > T$. Thus we get for $N \to \infty$
	\[
		\|u(T)\|_{H^s} \leq \|u_0\|_{H^s} e^{C \int_0^T \|\nabla u(t)\|_{L^\infty} \;dt},\; 0 \leq T < T^\ast(u_0).
	\]
Therefore if $T^\ast(u_0) < \infty$ we get by Lemma \ref{blowup_hs} that $\int_0^{T^\ast(u_0)} \|\nabla u(t)\|_{L^\infty} \;dt = \infty$.
\end{proof}

\begin{Coro}\label{coro_regularity_preserving}
	Let $s' > s$ and $u_0 \in H^{s'}_\omega(\R^{2n};\R^{2n})$. Then the solution to \eqref{symplectic_euler} in $H^{s'}$ coincides with the one in $H^s$. In particular there is neither a loss nor a gain of regularity. 
\end{Coro}

\begin{proof}
	Let $u$ be the solution to \eqref{symplectic_euler} in $H^s$ on its maximal time of existence $[0,T_s^\ast(u_0))$. Suppose the solution to \eqref{symplectic_euler} in $H^{s'}$ exists up to a time $T_{s'}^\ast(u_0) < T_s^\ast(u_0)$. Since $\|u(t)\|_{H^s}$ is bounded on $[0,T_{s'}^\ast(u_0)]$ we get by the Sobolev Imbedding Theorem that $\|\nabla u(t)\|_{L^\infty}$ is bounded on $[0,T_{s'}^\ast(u_0)]$, which contradicts Lemma \ref{blowup_sup}. Hence $T_{s'}^\ast(u_0)=T_s^\ast(u_0)$ and the solutions coincide. Suppose we start a solution with $u_0 \in H^s_\omega(\R^{2n};\R^{2n})$ and at some time $T > 0$ we have $u(T) \in H^{s'}_\omega(\R^{2n};\R^{2n})$. Solving \eqref{symplectic_euler} backwards shows that $u_0 \in H^{s'}_\omega(\R^{2n};\R^{2n})$. So the regularity class of a solution to \eqref{symplectic_euler} is preserved.
\end{proof}

To show global well-posedness we will use that the symplectic divergence of $u$ is ``frozen'' into the flow of \eqref{symplectic_euler} similar to the vorticity as in the case of the Euler equations \eqref{euler} -- see \cite{lagrangian}. That the symplectic divergence of $u$ has this property follows from the symmetries of \eqref{first_order_ode} and Noether's Theorem as we will see in the next section. In this section we will verify this directly. We denote by $\nabla_\omega=(\partial_2,-\partial_1,\ldots,\partial_{2n},\partial_{2n-1})$ the symplectic nabla operator, for $H:\R^{2n} \to \R$ we denote its symplectic gradient by $\nabla_\omega H=(\partial_2 H,-\partial_1 H,\ldots,\partial_{2n}H,-\partial_{2n-1}H)$ and for $u=(u_1,\ldots,u_{2n}) \in H_\omega(\R^{2n};\R^{2n})$ the symplectic divergence of $u$ is given by $\nabla_\omega \cdot u=\partial_2 u_1 - \partial_1 u_2 + \ldots + \partial_{2n} u_{2n-1} - \partial_{2n-1} u_{2n}$. Note that $\nabla_\omega$ is defined by the requirement 
\[
	dH(X)=(\nabla_\omega H) \cdot \omega \cdot X, 
\]
for all $H:\R^{2n} \to \R$ and $X \in \R^{2n}$ a column vector. Here $dH$ denotes the differential of $H$. Note that $\nabla_\omega H$ is always in $H_\omega$ since $P(\nabla_\omega H)=D\nabla H-(D\nabla H)^\top=0$. This implies for $z \in H_\omega^\perp$
\[
	\langle z,\nabla_\omega H \rangle_{L^2}=-\langle \nabla_\omega \cdot z,H \rangle_{L^2}=0.
\]
Since this holds for all $H:\R^{2n} \to \R$ we get that $\nabla_\omega \cdot z=0$. In particular we have $\nabla_\omega \cdot B(u)=0$. The next lemma tells us that $u \in H_\omega$ is essentially controlled by $\nabla_\omega \cdot u$.

\begin{Lemma}\label{symplectic_divergence}
Let $n \geq 1$ and $s > 2n/2+1$. We have
\[
	u=-(-\Delta)^{-1} \nabla_\omega (\nabla_\omega \cdot u)
\]
for all $u \in H^s_\omega(\R^{2n};\R^{2n})$.
\end{Lemma}

\begin{proof}
	We denote by $\mathcal R_j,j=1,\ldots,2n$, the Riesz operators, i.e. $\mathcal R_j=\partial_j (-\Delta)^{-1/2}$ is the Fourier multiplier operator with symbol $-i\xi_j/|\xi|$ for $\xi=(\xi_1,\ldots,\xi_{2n}) \in \R^{2n}$. Note that the Riesz operators are bounded on the Sobolev spaces. Let $u=(u_1,\ldots,u_{2n}) \in H^s_\omega(\R^{2n};\R^{2n})$. We will first prove that there is $v \in H^s(\R^{2n})$ s.t. $u=(\mathcal R_2 v,-\mathcal R_1v,\ldots,\mathcal R_{2n}v,-\mathcal R_{2n-1}v)$. Note that this means $u=(-\Delta)^{-1/2} \nabla_\omega v$. So the only possible Ansatz for $v$ is $v=-(-\Delta)^{-1/2} \nabla_\omega \cdot u=-\mathcal R_2 u_1 + \mathcal R_1 u_2 - \ldots - \mathcal R_{2n} u_{2n-1} + \mathcal R_{2n-1} u_{2n}$. Since $u \in H_\omega^s(\R^{2n};\R^{2n})$, i.e. $P(u)=0$, we have
	\begin{align}\label{R_matrix}
	\omega^\top \cdot \left(\begin{matrix} \mathcal R_1 u_1 & \cdots &\mathcal R_{2n} u_1 \\ \vdots & \ddots & \vdots \\ \mathcal R_1 u_{2n} & \cdots & \mathcal R_{2n} u_{2n} \end{matrix}\right)=\left(\begin{matrix} \mathcal R_1 u_1 & \cdots &\mathcal R_{2n} u_1 \\ \vdots & \ddots & \vdots \\ \mathcal R_1 u_{2n} & \cdots & \mathcal R_{2n} u_{2n} \end{matrix}\right)^\top \cdot \omega.
\end{align}
	For $j=1,\ldots,n$, the $2j-1$-th row in \eqref{R_matrix} is
\[
	(-\mathcal R_1 u_{2j},\ldots,-\mathcal R_{2n} u_{2j})=(-\mathcal R_{2j-1} u_2,\mathcal R_{2j-1} u_1,\ldots,-\mathcal R_{2j-1} u_{2n}, \mathcal R_{2j-1} u_{2n-1}).	
\]
Thus we get for $v= \mathcal R_1 u_2 -\mathcal R_2 u_1 + \ldots + \mathcal R_{2n-1} u_{2n}- \mathcal R_{2n} u_{2n-1} $
	\[
		-\mathcal R_{2j-1}v=-\mathcal R_1^2 u_{2j}-\ldots-\mathcal R_{2n}^2 u_{2j} = u_{2j}.
	\]
	Similarly we get $\mathcal R_{2j}v= u_{2j-1}$ for $j=1,\ldots,n$. Thus $u=(-\Delta)^{-1/2} \nabla_\omega v=-(-\Delta)^{-1} \nabla_\omega (\nabla_\omega \cdot u)$ holds. 
\end{proof}

Let us apply the symplectic divergence to \eqref{symplectic_euler}. We get
\begin{align*}
	\partial_t(\nabla_\omega \cdot u)+\nabla_\omega \cdot [(u \cdot \nabla) u]=0.
\end{align*}
We have
\begin{align*}
	\nabla_\omega \cdot [(u \cdot \nabla) u]=&(u \cdot \nabla) (\nabla_\omega \cdot u)+\sum_{k=1}^{2n} \partial_2 u_k \partial_k u_1 - \sum_{k=1}^{2n} \partial_1 u_k \partial_k u_2 + \ldots \\
	&+ \sum_{k=1}^{2n} \partial_{2n} u_k \partial_k u_{2n-1} - \sum_{k=1}^{2n} \partial_{2n-1} u_k \partial_k u_{2n}\\
	=&(u \cdot \nabla) (\nabla_\omega \cdot u) + \operatorname{tr}(\omega \cdot du^\top \cdot du^\top),
\end{align*}
where $\operatorname{tr}(A)$ denotes the trace of a matrix $A$. For the trace term we have
\begin{align*}
	\operatorname{tr}(\omega \cdot du^\top \cdot du^\top)&=\operatorname{tr}(du^\top \cdot du^\top \cdot \omega)=\operatorname{tr}(du^\top \cdot \omega^\top \cdot du)=0,
\end{align*}
where we use the cyclic property of the trace operator $\operatorname{tr}(ABC)=\operatorname{tr}(BCA)$, $P(u)=\omega^\top \cdot du-du^\top \cdot \omega=0$ and that the trace of the skew symmetric matrix $du^\top \cdot \omega^\top \cdot du$ is $0$. So the symplectic divergence of \eqref{symplectic_euler} reads as
\[
	\partial_t (\nabla_\omega \cdot u)+(u \cdot \nabla) (\nabla_\omega \cdot u)=0.
\]
Thus denoting by $\varphi$ the flow of $u$ we have $\partial_t [(\nabla_\omega \cdot u) \circ \varphi]=0$. In other words $\nabla_\omega \cdot u$ is frozen into the flow of $u$. We can write this as
\begin{equation}\label{frozen}
	\nabla_\omega \cdot u(t)=(\nabla_\omega \cdot u_0) \circ \varphi(t)^{-1}.
\end{equation}
In particular since $\varphi$ is volume preserving we have that $\|\nabla_\omega \cdot u(t)\|_{L^p}=\|\nabla_\omega \cdot u_0\|_{L^p}$ for $2 \leq p \leq \infty$. By Lemma \ref{symplectic_divergence} we have a connection between $\nabla u$ and $\nabla_\omega \cdot u$ via a singular integral operator. That is not sufficient to bound $\|\nabla u\|_{L^\infty}$ in terms of $\|\nabla_\omega \cdot u\|_{L^\infty}$. But we have similar as in \cite{bkm}

\begin{Lemma}\label{lemma_log}
	Let $n \geq 1$ and $s > 2n/2+1$. Supppose that $T(D)$ is a Fourier multiplier operator with a bounded symbol $T(\cdot) \in C^\infty(\R^{2n} \smallsetminus \{0\};\mathbb C)$ which is homogeneous of degree $0$. Then there is $C > 0$ s.t.
\[
	\|T(D)f\|_{L^\infty} \leq C+C \|f\|_{L^2} +C \|f\|_{L^\infty}+ C \|f\|_{L^\infty} \ln(1+\|f\|_{H^{s-1}})
\]
	for all $f \in H^{s-1}(\R^{2n})$.
\end{Lemma}

\begin{proof}
	Let $f \in H^{s-1}(\R^{2n})$. We will use a dyadic partition to get the estimate. Consider the one given in \cite{chemin_book}. There are smooth radial functions $\theta,\eta:\R^{2n} \to [0,1]$, where $\theta$ is supported in $\overline{B_{4/3}(0)}$ and $\eta$ is supported in $\overline{B_{8/3}(0)} \smallsetminus B_{3/4}(0)$. Here $B_r(0) \subset \R^{2n}$ denotes the open ball of radius $r > 0$ around $0$. Moreover we have
\[
	\theta(\xi)+\sum_{j \geq 0} \eta(2^{-j} \xi)=1,\;\xi \in \R^{2n}.
\]
	So we can write in $H^{s-1}$
	\[
		T(D)f=T(D)\theta(D) f + \sum_{j \geq 0} T(D)\eta(2^{-j}D)f.
	\]
We have
	\[
		(T(D)\theta(D)f)(x)=\mathcal F^{-1}[T(\xi)\theta(\xi)\hat f(\xi)](x)=\frac{1}{(2\pi)^n} \int_{\R^{2n}} e^{ix\cdot \xi} T(\xi) \theta(\xi) \hat f(\xi) \;d\xi.
	\]
Since the domain of integration is compact there is $C > 0$ s.t.
	\[
		\|T(D)\theta(D)f\|_{L^\infty} \leq C \|\hat f\|_{L^2}=C \|f\|_{L^2}.
	\]
For the other part we write
	\[
		\sum_{j \geq 0} T(D)\eta(2^{-j}D)f=\sum_{j=0}^N T(D)\eta(2^{-j}D)f + \sum_{j \geq N+1} T(D)\eta(2^{-j}D)f
	\]
	for some $N \geq 0$ to be determined. Note that we have $T(D)\eta(2^{-j}D)=T(2^{-j}D)\eta(2^{-j}D)$ since $T(\cdot)$ is homogeneous of degree $0$. Thus we can estimate
\begin{align*}
	\|T(D)\eta(2^{-j}D)f\|_{L^\infty} &\leq \|\mathcal F^{-1}[T(2^{-j}\xi)\eta(2^{-j}\xi)\hat f(\xi)]\|_{L^\infty}\\
	&\leq \|\mathcal F^{-1}[T(2^{-j}\xi)\eta(2^{-j}\xi)] \ast f]\|_{L^\infty}\\
	&\leq \|\mathcal F^{-1}[T(2^{-j}\xi)\eta(2^{-j}\xi)]\|_{L^1} \|f\|_{L^\infty}\\
	&=\|\mathcal F^{-1}[T(\xi)\eta(\xi)]\|_{L^1} \|f\|_{L^\infty}.
\end{align*}
	Note that $T(\xi)\eta(\xi)$ is smooth and compactly supported. This implies that $\mathcal F^{-1}[T(\xi)\eta(\xi)]$ is integrable. Thus we have
	\begin{align*}
		\left\| \sum_{j=0}^N T(D)\eta(2^{-j}D)f \right\|_{L^\infty} &\leq (N+1) \|\mathcal F^{-1}[T(\xi)\eta(\xi)]\|_{L^1} \|f\|_{L^\infty}\\
		&\leq C(N+1) \|f\|_{L^\infty}
	\end{align*}
Consider now
	\begin{align*}
		\left(\sum_{j \geq N+1} T(D)\eta(2^{-j}D)f\right)(x)=\frac{1}{(2\pi)^n} \int_{\R^{2n}} e^{ix \cdot \xi} \left(\sum_{j \geq N+1} T(\xi)\eta(2^{-j}\xi)\right) \hat f(\xi)\;d\xi. 
	\end{align*}
We have by the Cauchy-Schwarz inequality
	\begin{align*}
		\left|\left(\sum_{j \geq N+1} T(D)\eta(2^{-j}D)f\right)(x)\right| &\leq C \int_{|\xi| \geq 2^N} |\hat f(\xi)| \;d\xi\\
		&\leq C \int_{|\xi| \geq 2^N} |\xi|^{-(s-1)} |\xi|^{s-1} |\hat f(\xi)| \;d\xi\\
		&\leq C 2^{-\delta N} \|f\|_{H^{s-1}},
	\end{align*}
where $\delta=(s-1)-n>0$. We conclude
	\[
		\left\| \sum_{j \geq 0} T(D)\eta(2^{-j}D)f \right\|_{L^\infty} \leq C (N+1) \|f\|_{L^\infty} + C 2^{-\delta N} \|f\|_{H^{s-1}}.
	\]
	Choosing $N \geq 0$ s.t. $\frac{\ln(1+\|f\|_{H^{s-1}})}{\delta \ln(2)} \leq N < \frac{\ln(1+\|f\|_{H^{s-1}})}{\delta \ln(2)}+1$ we arrive at
\[
	\left\| \sum_{j \geq 0} T(D)\eta(2^{-j}D)f \right\|_{L^\infty} \leq C \|f\|_{L^\infty} \ln(1+\|f\|_{H^{s-1}}) + C \|f\|_{L^\infty} + C.
\]
Putting the estimates together we arrive at
	\begin{align*}
		\|T(D)f\|_{L^\infty} &\leq \|T(D)\theta(D) f\|_{L^\infty} + \left\|\sum_{j \geq 0} T(D)\eta(2^{-j}D)f\right\|_{L^\infty}\\
		&\leq C+C \|f\|_{L^2} +C \|f\|_{L^\infty}+ C \|f\|_{L^\infty} \ln(1+\|f\|_{H^{s-1}}),
	\end{align*}
for some $C > 0$. This finishes the proof.
\end{proof}

By Lemma \ref{symplectic_divergence} we have for $u \in H^s_\omega(\R^{2n};\R^{2n})$, $n \geq 1$ and $s > 2n/2+1$,
\[
	\nabla u=-\nabla (-\Delta)^{-1} \nabla_\omega (\nabla_\omega \cdot u).
\]
Thus we get by Lemma \ref{lemma_log} 
\[
	\|\nabla u\|_{L^\infty} \leq C+C \|\nabla_\omega \cdot u\|_{L^2} +C \|\nabla_\omega \cdot u\|_{L^\infty}+ C \|\nabla_\omega \cdot u\|_{L^\infty} \ln(1+\|\nabla_\omega \cdot u\|_{H^{s-1}}).
\]
Using this we can prove the global well-posedness of \eqref{symplectic_euler}.

\begin{proof}[Proof of Theorem \ref{th_gwp}]
	Let $u_0 \in H^s_\omega(\R^{2n};\R^{2n})$ and $u(t),0 \leq t < T^\ast(u_0)$ be the corresponding solution to \eqref{symplectic_euler}. In the proof of Lemma \ref{blowup_sup} we established
	\[
		\|u(T)\|_{H^s} \leq \|u_0\|_{H^s} e^{C \int_0^T \|\nabla u(t)\|_{L^\infty} \;dt},\; 0 \leq T < T^\ast(u_0).
	\]
Adding $1$ and applying $\ln$ we get
	\begin{align*}
		&\ln(1+\|\nabla_\omega \cdot u(T)\|_{H^{s-1}}) \leq \ln(1+C\|u_0\|_{H^s})+ C \int_0^T \|\nabla u(t) \|_{L^\infty} \;dt\\
		&\leq \ln(1+C\|u_0\|_{H^s}) + CT (1+\|\nabla_\omega \cdot u_0\|_{L^2} + \|\nabla_\omega \cdot u_0\|_{L^\infty})\\
		&+C \|\nabla_\omega \cdot u_0\|_{L^\infty} \int_0^T \ln(1+\|\nabla_\omega \cdot u(t)\|_{H^{s-1}}) \;dt,\; 0 \leq T < T^\ast(u_0).
	\end{align*}
By Gr\"onwall's inequality we see that $\|\nabla_\omega \cdot u(t)\|_{H^{s-1}}$ stays bounded for finite $T$. This in turn implies that $\|\nabla u(t)\|_{H^{s-1}}$ stays bounded and with that we get by using the Sobolev Imbedding Theorem the boundedness of $\|\nabla u(t)\|_{L^\infty}$ for finite $T$. Lemma \ref{blowup_sup} implies now that $T^\ast(u_0)=\infty$, i.e. that the solution $u$ exists for all times.
\end{proof}

\section{The symplectomorphism group $\Ds^s_\omega(\R^{2n})$} \label{section_symplectomorphisms}

	Throughout this section we assume $n \geq 1$ and $s > 2n/2+1$. The goal of this section is to prove that $\Ds^s_\omega(\R^{2n}) \subset \Ds^s(\R^{2n})$ is a closed analytic submanifold. Recall that we've denoted by $U \subset H^s(\R^{2n};\R^{2n})$ the set of initial values $u_0$ for which the solution to \eqref{first_order_ode} exists beyond time $T=1$. With this we define the exponential map $\exp$ as
\[
	\exp:U \subset H^s(\R^{2n};\R^{2n}) \to \Ds^s(\R^{2n}),\;u_0 \mapsto \varphi(1;u_0),
\]
i.e. $\exp$ is the $\varphi$ component of the time $T=1$ solution map of \eqref{first_order_ode}. By analytic dependence on the initial value we get that $\exp$ is analytic. Moreover we conclude from $\varphi(t;u_0)=\varphi(1;tu_0)$ that $0 \leq t < T^\ast(u_0)$ iff $tu_0 \in U$ and $\varphi(t;u_0)=\exp(tu_0)$ for $0 \leq t < T^\ast(u_0)$. This means we can construct all solutions to \eqref{first_order_ode} with the exponential map $\exp$. The derivative of $\exp$ at $0$ in direction of $w \in H^s(\R^{2n};\R^{2n})$ is given by
\[
	\left.\frac{d}{dt}\right|_{t=0} \exp(tw)=\left.\frac{d}{dt}\right|_{t=0} \varphi(t;w)=\varphi_t(0;w)=w.
\]
This means that the derivative at $0$
\[
d_0\exp:H^s(\R^{2n};\R^{2n}) \to H^s(\R^{2n};\R^{2n}),\;w \mapsto w,
\]
is the identity map. Thus we can apply the Inverse Function Theorem. Before doing that we prove the following commutator estimate.

\begin{Lemma}\label{lemma_commutator}
	There is a constant $C > 0$ s.t.
	\[
		\left\| [u \cdot \nabla,\mathcal R_j] f\right\|_{L^2} \leq C \|u\|_{H^s} \|f\|_{L^2}
	\]
	for all $1 \leq j \leq 2n,f \in L^2(R^{2n}),u \in H^s(\R^{2n};\R^{2n})$. Here $\mathcal R_j=\dfrac{\partial_j}{(-\Delta)^{1/2}},j=1,\ldots,2n$, are the Riesz operators and $[A,B]=AB-BA$ denotes the commutator of operators.
\end{Lemma}

\begin{proof}
	Let $1 \leq j \leq 2n$. The Riesz operator $\mathcal R_j$ is a Fourier multiplier operator with symbol $-i\xi_j/|\xi|,\xi \in \R^{2n}$. The Fourier Transform of $[u \cdot \nabla,\mathcal R_j] f$ is given by
\begin{align*}
	\mathcal F\left[[u \cdot \nabla,\mathcal R_j] f\right](\xi)=\int_{\R^{2n}} \hat u(\xi-\eta) \cdot \eta \left(\frac{\xi_j}{|\xi|}-\frac{\eta_j}{|\eta|}\right) \hat f(\eta) \;d\eta. 
\end{align*}
Thus we can estimate
	\begin{align*}
		\left|\mathcal F\left[[u \cdot \nabla,\mathcal R_j] f\right](\xi)\right| \leq \int_{\R^{2n}} |\hat u(\xi-\eta)| |\eta| \left|\frac{\xi}{|\xi|}-\frac{\eta}{|\eta|}\right| |\hat f(\eta)| \;d\eta. 
\end{align*}
We claim that 
	\[
		|\eta| \left|\frac{\xi}{|\xi|}-\frac{\eta}{|\eta|}\right| \leq 2 |\xi-\eta|,\;\xi,\eta \in \R^{2n} \smallsetminus \{0\}.
	\]
Rewritten this reads as
	\begin{equation}\label{vector_ineq}
		\left|\frac{\xi}{|\xi|}-\frac{\eta}{|\eta|}\right| \leq 2 \left|\frac{\xi}{|\eta|}-\frac{\eta}{|\eta|}\right|,\;\xi,\eta \in \R^{2n} \smallsetminus \{0\}.
	\end{equation}
	If the angle $\angle (\xi,\eta) \in [0,\pi]$ between $\xi,\eta \in \R^{2n} \smallsetminus \{0\}$ satisfies $\angle (\xi,\eta) \geq \pi/2$ we then have $|\xi/|\eta|-\eta/|\eta|| \geq 1$ and \eqref{vector_ineq} holds by the triangle inequality. For $\angle(\xi,\eta) \in [0,\pi/2)$ we have
\[
	\left|\frac{\xi}{|\eta|}-\frac{\eta}{|\eta|}\right| \geq \sin(\angle(\xi,\eta)),
\]
	since the orthogonal projection of $\eta/|\eta|$ in direction of $\xi/|\eta|$ is $\sin(\angle(\xi,\eta))$. On the other hand we have
\[
	\left|\frac{\xi}{|\xi|}-\frac{\eta}{|\eta|}\right| = 2 \sin(\frac{1}{2} \angle(\xi,\eta)).
\]
Due to the trigonometric identity
	\[
		\sin(\angle(\xi,\eta))=2 \sin(\frac{1}{2} \angle(\xi,\eta)) \cos(\frac{1}{2} \angle(\xi,\eta))
	\]
	we have $2 \sin(\frac{1}{2} \angle(\xi,\eta)) \leq 2 \sin(\angle(\xi,\eta))$ and the inequality \eqref{vector_ineq} follows. So we have
	\[
		\left|\mathcal F\left[[u \cdot \nabla,\mathcal R_j] f\right](\xi)\right| \leq 2 \int_{\R^{2n}} |\hat u(\xi-\eta)| |\xi-\eta| |\hat f(\eta)| \;d\eta. 
		\]
Using Young's inequality for convolutions we get
	\[
		\|\mathcal F\left[[u \cdot \nabla,\mathcal R_j] f\right]\|_{L^2} \leq 2 \| |\hat u(\xi)| |\xi| \|_{L^1(d\xi)} \|\hat f\|_{L^2}.
	\]
Note that
	\[
		\int_{\R^{2n}} |\xi| |\hat u(\xi)| \;d\xi = \int_{\R^{2n}} \frac{1}{(1+|\xi|^2)^{(s-1)/2}} (1+|\xi|^2)^{(s-1)/2} |\xi| |\hat u(\xi)| \;d\xi.
	\]
Thus applying Cauchy-Schwarz we get
	\[
		\| |\hat u(\xi)| |\xi| \|_{L^1(d\xi)} \leq C \|u\|_{H^s}
	\]
for some constant $C > 0$ as $s-1 > 2n/2$. Plancherel's Theorem now gives us the needed estimate
	\[
		\left\| [u \cdot \nabla,\mathcal R_j] f\right\|_{L^2} \leq C \|u\|_{H^s} \|f\|_{L^2}.
	\]
\end{proof}

\begin{Lemma}\label{exp_chart}
	There are open neighborhoods $V \subset H^s(\R^{2n};\R^{2n})$ of $0$ resp. $W \subset \Ds^s(\R^{2n})$ of $\text{id}$ with $V \subset U$ s.t. $\left. \exp \right|_V:V \to W$ is a diffeomorphism and
\[
	\exp(V \cap H^s_\omega(\R^{2n};\R^{2n}))=W \cap \Ds^s_\omega(\R^{2n}).
\]
\end{Lemma}

\begin{proof}
	By the Inverse Function Theorem we fix $R > 0$ s.t. $\left. \exp \right|_{B_R(0)}$ is a diffeomorphism onto its image. Here $B_R(0) \subset H^s(\R^{2n};\R^{2n})$ is the open ball of radius $R$ with center $0$, i.e. $B_R(0)=\{w \in H^s(\R^{2n};\R^{2n}) \;|\; \|w\|_{H^s} < R\}$. First we prove
\[
	\exp(B_R(0) \cap H^s_\omega(\R^{2n};\R^{2n})) \subset \Ds^s_\omega(\R^{2n}).
\]
	Let $u_0 \in B_R(0) \cap H^s_\omega(\R^{2n};\R^{2n})$. We define $\varphi(t)=\exp(tu_0),0 \leq t \leq 1$ and with that $u(t)=\varphi_t(t) \circ \varphi(t)^{-1},0 \leq t \leq 1$. We know that $u$ solves \eqref{symplectic_euler_new} since it starts in $H^s_\omega(\R^{2n};\R^{2n})$. Thus we have $L_{u(t)}\omega=0,0 \leq t \leq 1$. Consider now
\[
	\frac{d}{dt} \varphi(t)^\ast \omega=\varphi(t)^\ast L_{u(t)} \omega=0. 
\]
	Since $\varphi(0)^\ast \omega=\omega$ we get $\varphi(1)^\ast \omega=\omega$. Thus $\varphi(1)=\exp(u_0)$ is in $\Ds^s_\omega(\R^{2n})$. Next we prove that there is $0 < \delta \leq R$ s.t.
\[
	\exp(B_\delta(0) \smallsetminus H^s_\omega(\R^{2n};\R^{2n})) \subset \Ds^s(\R^{2n}) \smallsetminus \Ds^s_\omega(\R^{2n}).
\]
Note that $\varphi^\ast \omega=\omega$ written in matrix form is
	\[
		(d_x \varphi)^\top \cdot \omega \cdot d_x \varphi=\omega,\;x \in \R^{2n}.
	\]
If we take the $t$ derivative of the left hand side we get
	\begin{align*}
		\frac{d}{dt} ((d\varphi)^\top \cdot \omega \cdot d\varphi) &= (d\varphi_t)^\top \cdot \omega \cdot d\varphi+ (d\varphi)^\top \cdot \omega \cdot d\varphi_t\\
		&=(d\varphi)^\top \cdot (du)^\top \circ \varphi \cdot \omega \cdot d\varphi+(d\varphi)^\top \cdot \omega \cdot du \circ \varphi \cdot d\varphi\\
		&=-(d\varphi)^\top \cdot ( \omega^\top \cdot du-(du)^\top \cdot \omega) \circ \varphi \cdot d\varphi\\
		&=- (d\varphi)^\top \cdot P(u) \circ \varphi \cdot d\varphi. 
	\end{align*}
	where we used $d\varphi_t=du \circ \varphi \cdot d\varphi$ and $P(u)=\omega^\top \cdot du-(du)^\top \cdot \omega$. For $u_0 \in B_R(0)$ denote by $\varphi(t)=\exp(tu_0),0 \leq t \leq 1$ and $u(t)=\varphi_t(t) \circ \varphi(t)^{-1},0 \leq t \leq 1$. We have
\begin{align*}
	(d\varphi(1))^\top \cdot \omega \cdot d\varphi(1)-\omega&=\int_0^1 \frac{d}{dt} ((d\varphi(t))^\top \cdot \omega \cdot d\varphi(t)) \;dt\\
	&=-\int_0^1 (d\varphi(t))^\top \cdot P(u(t)) \circ \varphi(t) \cdot d\varphi(t) \;dt
\end{align*}
Now we calculate
	\begin{align*}
		&\frac{d}{dt} [(d\varphi)^\top \cdot P(u) \circ \varphi \cdot d\varphi]=(d\varphi)^\top \cdot (du)^\top \circ \varphi \cdot P(u) \circ \varphi \cdot d\varphi\\
		&+(d\varphi)^\top \cdot [P(u_t)+(u \cdot \nabla)P(u)] \circ \varphi \cdot d\varphi+(d\varphi)^\top \cdot P(u) \cdot du \circ \varphi \cdot d\varphi.
	\end{align*}
	We need to estimate the above in terms of $P(u)$. The only term causing difficulties is $P(u_t)+(u \cdot \nabla) P(u)$. We have by \eqref{symplectic_euler_new}
\begin{align*}
	P(u_t)=& -P((u \cdot \nabla)u)+PB(u)=\\
	&-P((u \cdot \nabla)u)-\frac{1}{2} \Delta^{-1} (1-\chi(D)) PP^\ast P_H(u)-\frac{1}{2} \chi(D) \Delta^{-1} PP^\ast P_L(u).
\end{align*}
	Replacing $P_H(u)=P((u \cdot \nabla)u)-(u \cdot \nabla) P(u)$ and $P_L(u)=P((u \cdot \nabla)u)-(u \cdot \nabla) P(u)+Q(u)$ we can write
\begin{align*}
	P(u_t)=& -P((u \cdot \nabla)u)
	-\frac{1}{2} \Delta^{-1} PP^\ast P((u \cdot \nabla)u)\\
	&+ \frac{1}{2} \Delta^{-1} PP^\ast ((u \cdot \nabla) P(u)) -\frac{1}{2} \Delta^{-1} PP^\ast \chi(D)Q(u).
\end{align*}
	Note that $\Delta^{-1} PP^\ast$ is a linear combination of operators $\mathcal R_i \mathcal R_j,1 \leq i,j \leq 2n$. In particular it is bounded on the Sobolev spaces. If $P^\ast P(w)=0$ then $P(w)=0$ as one can see from $\langle P^\ast P(w),w\rangle_{L^2} = \langle P(w),P(w) \rangle_{L^2}$. If we apply $P^\ast$ to $-P((u \cdot \nabla)u) -\frac{1}{2} \Delta^{-1} PP^\ast P((u \cdot \nabla)u)$ we get
\begin{align*}
	-P^\ast P((u \cdot \nabla) u ) + P^\ast P((u \cdot \nabla)u)=0,
\end{align*}
	where we used $-\frac{1}{2} P^\ast P P^\ast = \Delta P^\ast$. So we conclude
\begin{align*}
	P(u_t)= \frac{1}{2} \Delta^{-1} PP^\ast ((u \cdot \nabla) P(u)) -\frac{1}{2} \Delta^{-1} PP^\ast \chi(D)Q(u).
\end{align*}
We write this as
\begin{align*}
	&P(u_t)= \\
	&-\frac{1}{2} [u \cdot \nabla,\Delta^{-1} PP^\ast] P(u) + \frac{1}{2} (u \cdot \nabla) \Delta^{-1}P P^\ast P(u) -\frac{1}{2} \Delta^{-1} PP^\ast \chi(D)Q(u).
\end{align*}
	Applying $P^\ast$ to $\frac{1}{2} \Delta^{-1} P P^\ast P(u)+P(u)$ gives $0$ where we use again $-\frac{1}{2}P^\ast P P^\ast=\Delta P^\ast$. In particular we have
\[
	(u \cdot \nabla) \frac{1}{2} \Delta^{-1}P P^\ast P(u) + (u \cdot \nabla) P(u)=0.
\]
So we end up with
\begin{align*}
	P(u_t)+(u \cdot \nabla)P(u)= 
	-\frac{1}{2} [u \cdot \nabla,\Delta^{-1} PP^\ast] P(u)  -\frac{1}{2} \Delta^{-1} PP^\ast \chi(D)Q(u).
\end{align*}
Note that the first term on the right hand side consists of linear combinations of expressions of the form $[u \cdot \nabla, \mathcal R_i \mathcal R_j] f,1 \leq i,j \leq 2n$, and $f$ is a component of $P(u)$. For such expressions we can write
\[
	[u \cdot \nabla,\mathcal R_i \mathcal R_j] f = [u \cdot \nabla,\mathcal R_i] \mathcal R_j f + \mathcal R_i [u \cdot \nabla,\mathcal R_j] f.
\]
Applying Lemma \ref{lemma_commutator} gives
\[
	\left\|-\frac{1}{2} [u \cdot \nabla,\Delta^{-1} PP^\ast] P(u)\right\|_{L^2} \leq C \|u\|_{H^s} \|P(u)\|_{L^2},
\]
for some $C > 0$ independent of $u$. Using \eqref{Q_estimate} with $f=u$, $g=\operatorname{div} u$ and $h \in L^2$ a test function we get
\[
	\|\chi(D)Q\|_{L^2} \leq C \|u\|_{H^s} \|\operatorname{div}u\|_{L^2} \leq C \|u\|_{H^s} \|P(u)\|_{L^2}
\]
for some $C > 0$. Putting these two estimates together we get
\[
	\|P(u_t)+(u \cdot \nabla)P(u)\|_{L^2} \leq C \|u\|_{H^s} \|P(u)\|_{L^2}
\]
for some $C > 0$. From \eqref{p_ineq} we have
\[
	\frac{d}{dt} \|P(u)\|_{L^2} \leq C \|u\|_{H^s} \|P(u)\|_{L^2}.
\]
Since $u_0 \mapsto u$ is continuous we can control the size of $\|u\|_{H^s}$ by taking $u_0$ small enough. Choosing $0 < \delta_1 \leq R$ small enough we can ensure due to Gr\"onwall's inequality that
\[
	\|P(u)\|_{L^2} \leq 2 \|P(u_0)\|_{L^2}
\]
for $u_0 \in B_{\delta_1}(0)$. Now we write
\begin{align*}
	(d\varphi(1))^\top \cdot \omega \cdot d\varphi(1)-\omega
	&=-\int_0^1 (d\varphi(t))^\top \cdot P(u(t)) \circ \varphi(t) \cdot d\varphi(t) \;dt\\
	&=P(u_0)-\int_0^1 \int_0^t \frac{d}{dt} [d\varphi^\top \cdot P(u) \circ \varphi \cdot d\varphi] \;dsdt.
\end{align*}
The latter inner integrand consists of expressions of the form $d\varphi^\top \cdot f \circ \varphi \cdot d\varphi$. Note that 
\[
	\Ds^s(\R^{2n}) \times L^2(\R^{2n}) \to L^2(\R^{2n}),\;(\varphi,f) \mapsto  d\varphi^\top \cdot f \circ \varphi \cdot d\varphi,
\]
is continuous. Moreover since it is linear in $f$ there is $0 < \delta_2 \leq \delta_1$ s.t. for some $C > 0$ we have
\[
	\left\|d\varphi^\top \cdot f \circ \varphi \cdot d\varphi\right\|_{L^2} \leq C \|f\|_{L^2}
\]
for all $f \in L^2(\R^{2n})$ and $\varphi \in \exp(B_{\delta_2}(0))$. Using this we get 
\begin{align*}
	&\left\|\frac{d}{dt} [(d\varphi)^\top \cdot P(u) \circ \varphi \cdot d\varphi]\right\|_{L^2} = \|d\varphi^\top \cdot (du)^\top \circ \varphi \cdot P(u) \circ \varphi \cdot d\varphi\\
	&+ d\varphi^\top \cdot [P(u_t)+(u \cdot \nabla)P(u)] \circ \varphi \cdot d\varphi+d\varphi^\top \cdot P(u) \cdot du \circ \varphi \cdot d\varphi\|_{L^2}\\
	&\leq C \|u\|_{H^s} \|P(u)\|_{L^2}.
\end{align*}
By taking $0 < \delta \leq \delta_2$ small enough we can ensure that
\[
	\|d\varphi(1)^\top \cdot \omega \cdot d\varphi(1)-\omega\|_{L^2} \geq \frac{1}{2} \|P(u_0)\|_{L^2}.
\]
So we arrive at $\exp(B_\delta(0) \smallsetminus H_\omega^s(\R^{2n};\R^{2n})) \subset \Ds^s(\R^{2n}) \smallsetminus \Ds^s_\omega(\R^{2n})$. Hence 
\[
	\exp(B_\delta(0) \cap H_\omega^s(\R^{2n};\R^{2n})) = \exp(B_\delta(0)) \cap \Ds^s_\omega(\R^{2n}).
\]
This finishes the proof.
\end{proof}

If we take $V,W$ as in Lemma \ref{exp_chart} we get a chart $\left(\left.\exp\right|_V\right)^{-1}:W \to V$ for $\Ds^s_\omega(\R^{2n})$ around $\text{id}$ with 
\[
	\left(\left.\exp\right|_V\right)^{-1}(W \cap \Ds^s_\omega(\R^{2n}))=V \cap H^s_\omega(\R^{2n};\R^{2n}).
\]
We can use this to get a chart around every element of $\Ds^s_\omega(\R^{2n})$. So we can prove Theorem \ref{th_submanifold}.

\begin{proof}[Proof of Theorem \ref{th_submanifold}]
That $\Ds^s_\omega(\R^{2n}) \subset \Ds^s(\R^{2n})$ is a closed subset follows from the continuity of
	\[
		G:\Ds^s(\R^{2n}) \to H^{s-1}(\R^{2n};\R^{2n \times 2n}_\text{skew}),\; \varphi \mapsto (d\varphi)^\top \cdot \omega \cdot d\varphi - \omega,
	\]
	and $\Ds^s_\omega(\R^{2n})=G^{-1}(\{0\})$. The symplectomorphisms $\Ds^s_\omega(\R^{2n})$ form a subgroup of $\Ds^s(\R^{2n})$ as one can see easily from the characterization $(d\varphi)^\top \cdot \omega \cdot d\varphi = \omega$. Now let $\psi \in \Ds^s_\omega(\R^{2n})$. We denote by
	\[
		R_\psi:\Ds^s(\R^{2n}) \to \Ds^s(\R^{2n}),\;\varphi \mapsto \varphi \circ \psi,
	\]
	the composition from the right with $\psi$. Note that $R_\psi$ is analytic since it is an affine map. Moreover it is invertible with inverse $R_\psi^{-1}=R_{\psi^{-1}}$. Let $V,W$ be as in Lemma \ref{exp_chart}. Then $(R_\psi \circ \left. \exp \right|_V)^{-1}:R_\psi(W) \to V$ provides a chart around $\psi$. This shows that $\Ds^s_\omega(\R^{2n})$ is a closed analytic submanifold of $\Ds^s(\R^{2n})$. 
\end{proof}

\section{Lagrangian mechanics}\label{section_lagrangian_mechanics}

Lemma \ref{exp_chart} tells us that the tangent space of $\Ds^s_\omega(\R^{2n})$ at $\text{id}$ is given by
\[
	T_\text{id}\Ds^s_\omega(\R^{2n})=H^s_\omega(\R^{2n};\R^{2n}).
\]
	For $\psi \in \Ds^s_\omega(\R^{2n})$ the affine diffeomorphism
\[
	R_\psi:\Ds^s(\R^{2n}) \to \Ds^s(\R^{2n}),\;\varphi \mapsto \varphi \circ \psi,
\]
	shows us that the tangent space of $\Ds^s_\omega(\R^{2n})$ at $\psi$ is given by
\[
	T_\psi \Ds^s_\omega(\R^{2n})=R_\psi(H^s_\omega(\R^{2n};\R^{2n})),
\]
	i.e. $w \in T_\psi \Ds^s_\omega(\R^{2n})$ is of the form $w=\tilde w \circ \psi$ for some $\tilde w \in H^s_\omega(\R^{2n};\R^{2n})$.\\

Now that we've a smooth structure for $\Ds^s_\omega(\R^{2n})$ we try to express \eqref{symplectic_euler} using the formalism of Lagrangian mechanics. We take
\[
	L:T\Ds^s_\omega(\R^{2n}) \to \R,\; (\varphi,\varphi_t) \mapsto L(\varphi,\varphi_t)=\frac{1}{2} \int_{\R^{2n}} \varphi_t^2 \;dx,
\]
as Lagrangian. Here $T\Ds^s_\omega(\R^{2n})$ denotes the tangent bundle of $\Ds^s_\omega(\R^{2n})$. With this we consider the action functional defined for $C^1$ paths $\varphi:[0,1] \to \Ds^s_\omega(\R^{2n})$ as
\[
	\mathcal L(\varphi)=\int_0^1 L(\varphi(t),\varphi_t(t)) \;dt.
\]
We fix $\psi_0,\psi_1 \in \Ds^s_\omega(\R^{2n})$ and denote
\[
	S_{\psi_0}^{\psi_1}=\{\varphi:[0,1] \to \Ds^s_\omega(\R^{2n}) \;|\; \varphi \text{ is a } C^2 \text{ path},\varphi(0)=\psi_0,\varphi(1)=\psi_1\}.
\]
The reason we take $C^2$ is for doing an integration by parts as we will see below. Hamilton's principle tells us that the system described by the Lagrangian $L(\varphi,\varphi_t)$ evolves as $\varphi \in S_{\psi_0}^{\psi_1}$ if $\varphi$ is a stationary value of the action $\mathcal L(\cdot)$ on $S_{\psi_0}^{\psi_1}$. To get an explicit equation we take a $C^1$ variation of $\varphi$ with the endpoints fixed, i.e. we take a
\[
	\tilde \varphi:[0,1] \times (-\delta,\delta) \to \Ds^s(\R^{2n}),\; (t,\varepsilon) \mapsto \tilde \varphi(t,\varepsilon),
\]
for some $\delta > 0$ s.t. $\tilde \varphi(\cdot,\varepsilon) \in S_{\psi_0}^{\psi_1}$ for all $\varepsilon \in (-\delta,\delta)$ and $\tilde \varphi(\cdot,0)=\varphi$. Doing then the variation gives
\begin{align*}
	\left. \frac{d}{d\varepsilon} \right|_{\varepsilon=0} \mathcal L(\tilde \varphi(\cdot,\varepsilon)) &= \int_0^1 \int_{\R^{2n}} \varphi_t(t) \cdot \left. \partial_\varepsilon\right|_{\varepsilon=0} \tilde \varphi_t(t,\varepsilon) \;dx \;dt\\
	&=-\int_0^1 \int_{\R^{2n}} \varphi_{tt}(t) \cdot \left. \partial_\varepsilon\right|_{\varepsilon=0} \tilde \varphi(t,\varepsilon) \;dx \;dt,
\end{align*}
where we use that $\left. \partial_\varepsilon\right|_{\varepsilon=0} \tilde \varphi(0,\varepsilon)=\left. \partial_\varepsilon\right|_{\varepsilon=0} \tilde \varphi(1,\varepsilon)=0$. Note that $\left. \partial_\varepsilon\right|_{\varepsilon=0} \tilde \varphi(t,\varepsilon) \in T_{\varphi(t)} \Ds^s_\omega(\R^{2n})$. Hence a necessary and sufficient condition that $\mathcal L(\cdot)$ is stationary on $S_{\psi_0}^{\psi_1}$ at $\varphi \in S_{\psi_0}^{\psi_1}$ is for all $0 \leq t \leq 1$
\begin{equation*}
	\int_{\R^{2n}} \varphi_{tt}(t) \cdot w \circ \varphi(t) \;dx=0,\;w \in H_\omega^s(\R^{2n};\R^{2n}).
\end{equation*}
This is the Euler-Lagrange equation for $L(\varphi,\varphi_t)$. Since $\varphi(t)$ is volume preserving we get
\[
	\varphi_{tt}(t) \circ \varphi(t)^{-1} \in H^s_\omega(\R^{2n};\R^{2n})^\perp,\; 0 \leq t \leq 1,
\]
where $\perp$ refers to the $L^2$ orthogonal complement. Thus we get for $u(t):=\varphi_t(t) \circ \varphi(t)^{-1} \in H^s_\omega(\R^{2n};\R^{2n}),0 \leq t \leq 1$,
\[
	\varphi_{tt}(t) \circ \varphi(t)^{-1}=u_t(t) + (u(t) \cdot \nabla) u(t) \in H^s_\omega(\R^{2n};\R^{2n})^\perp,\; 0 \leq t \leq 1.
\]
So $u$ solves \eqref{symplectic_euler}.\\
We have for $v \in H^s_\omega(\R^{2n};\R^{2n})$ and $\varepsilon \in \R$
\[
	L(\varphi \circ \exp(\varepsilon v),\varphi_t \circ \exp(\varepsilon v))=L(\varphi,\varphi_t).
\]
This symmetry of $L(\varphi,\varphi_t)$ implies by Noether's Theorem the following conservation law for a stationary value $\varphi$ of $\mathcal L(\varphi)$
\[
	\int_{\R^{2n}} \varphi_t \cdot \left.\partial_\varepsilon\right|_{\varepsilon=0} \varphi \circ \exp(\varepsilon v) \;dx=\int_{\R^{2n}} \varphi_t \cdot d\varphi \cdot v \;dx \equiv \text{const},\;0 \leq t \leq 1,
\] 
for all $v \in H^s_\omega(\R^{2n};\R^{2n})$. By taking $v=-\nabla_\omega H$ for $H \in C_c^\infty(\R^{2n})$ we get that $\nabla_\omega \cdot ((d\varphi)^\top \cdot \varphi_t)$ is conserved. The expression where the derivative hits the first factor are
\begin{align*}
	\sum_{k=1}^{2n} \partial_2 \partial_1 \varphi_k \partial_t \varphi_k - \partial_1 \partial_2 \varphi_k \partial_t \varphi_k + \ldots + \partial_{2n} \partial_{2n-1} \varphi_k \partial_t \varphi_k - \partial_{2n-1} \partial_{2n} \varphi_k \partial_t \varphi_k=0.
\end{align*}
So we conclude from the formula $\nabla_\omega \cdot z = \operatorname{tr}(dz \cdot \omega^\top)$ for $z \in \R^{2n}$
\[
	\nabla_\omega \cdot ((d\varphi)^\top \cdot \varphi_t)= \operatorname{tr}((d\varphi)^\top \cdot d\varphi_t \cdot \omega^\top).
\]
On the other hand we have for $u=\varphi_t \circ \varphi^{-1}$
\begin{align*}
	(\nabla_\omega \cdot u) \circ \varphi&=\operatorname{tr}(d\varphi_t \cdot (d\varphi)^{-1} \omega^\top)= \operatorname{tr}(d\varphi_t \cdot (d\varphi)^{-1} \cdot \omega^\top \cdot ((d\varphi)^{-1}))^\top \cdot (d\varphi)^\top)\\
	&=\operatorname{tr}(d\varphi_t \cdot \omega^\top \cdot (d\varphi)^\top)=\operatorname{tr}((d\varphi)^\top \cdot d\varphi_t \cdot \omega^\top)=\nabla_\omega \cdot ((d\varphi)^\top \cdot \varphi_t),
\end{align*}
where we use that $(d\varphi)^{-1}$ is a symplectic matrix and the cyclic property of the trace operator $\operatorname{tr}$. Thus we see that Noether's Theorem implies the conservation of $(\nabla_\omega \cdot u) \circ \varphi$.

\section{Nonuniform dependence}\label{section_nonuniform}

The goal of this section is to prove Theorem \ref{th_nonuniform}. Before we do that we have to do some preliminary work. Throughout this section we assume $n \geq 1$ and $s > 2n/2+1$. First of all note that due to the scaling property \eqref{scaling} we have for the time $T > 0$ solution map $\Phi_T$
\begin{equation}\label{solution_map}
	\Phi_T(u_0)=\frac{1}{T} \Phi(Tu_0),\; u_0 \in H^s_\omega(\R^{2n};\R^{2n}),
\end{equation}
where $\Phi$ is the time one solution map, i.e. $\Phi=\left. \Phi_\tau \right|_{\tau=1}$. Thus to prove Theorem \ref{th_nonuniform} it is sufficient to prove it in the special case $T=1$. We will need the following technical lemma concerning the exponential map
\[
	\exp:H^s_\omega(\R^{2n};\R^{2n}) \to \Ds^s(\R^{2n}).
\]

\begin{Lemma}\label{lemma_dexp}
	There is a dense subset $S \subset H^s_\omega(\R^{2n};\R^{2n})$ s.t. we have $S \subset H^{s+1}_\omega(\R^{2n};\R^{2n})$ and
\[
	d_w \exp \neq 0,\;w \in S.
\]
\end{Lemma}

\begin{proof} 
	First note that that $H^{s+1}_\omega(\R^{2n};\R^{2n}) \subset H^s_\omega(\R^{2n};\R^{2n})$ is dense. To see this we denote by $\chi_N(D)$ the Fourier multiplier operator with symbol the characteristic function of the closed ball in $\R^{2n}$ with center $0$ and radius $N > 0$. Now for $w \in H^s_\omega(\R^{2n};\R^{2n})$ we have $\chi_N(D)w \to w$ in $H^s$ and $\chi_N(D)w \in H^{s+1}_\omega(\R^{2n};\R^{2n})$. The reason is that $\chi_N(D)$ commmutes with the operator $P$ from \eqref{operator_p} which implies $P(\chi_N(D)w)=\chi_N(D)P(w)=0$.\\
	Suppose now that there is no dense subset $S$ with the properties in the statement of the lemma. Then due to the density of $H^{s+1}_\omega(\R^{2n};\R^{2n}) \subset H^s_\omega(\R^{2n};\R^{2n})$ there is a nonempty open set $V \subset H^s_\omega(\R^{2n};\R^{2n})$ s.t. $d_w \exp = 0,w \in V$. This implies that $\exp$ is constant on $V$. Since $\exp$ is analytic we get that $\exp$ is constant over its whole domain. This is clearly not possible since we haver e.g. $d_0 \exp =\text{id}_{H^s_\omega(\R^{2n};\R^{2n})}$.
\end{proof}

We have the following local Lipschitz estimate.

\begin{Lemma}\label{lemma_lipschitz}
	Let $\varphi \in \Ds^s(\R^{2n})$. Then there is a neighborhood $W \subset \Ds^s(\R^{2n})$ of $\varphi$ and $C > 0$ s.t.
\[
	\|f \circ \varphi_1-f \circ \varphi_2\|_{H^{s-1}} \leq C \|f\|_{H^s} \|\varphi_1-\varphi_2\|_{H^{s-1}},
\]
	for all $f \in H^s(\R^{2n})$ and $\varphi_1,\varphi_2 \in W$.
\end{Lemma}

\begin{proof}
	Let $\delta > 0$ be s.t. the ball $B_\delta(\varphi)$ with center $\varphi$ and radius $\delta$ is contained in $\Ds^s(\R^{2n})$ and s.t. we have for some $C > 0$
	\[
		\|g \circ \psi\|_{H^{s-1}} \leq C \|g\|_{H^{s-1}}
	\]
for all $\psi \in B_\delta(\varphi)$ and $g \in H^{s-1}(\R^{2n})$. This follows from the continuity of the composition together with linearity in $g$. By the Fundamental Theorem of Calculus we have for $f \in H^s(\R^{2n})$, $\varphi_1,\varphi_2 \in B_\delta(\varphi)$ pointwise in $x \in \R^{2n}$
\[
	f(\varphi_2(x))-f(\varphi_1(x))=\int_0^1 \nabla f(\varphi_1(x)+s(\varphi_2(x)-\varphi_1(x)) \cdot (\varphi_2(x)-\varphi_1(x)) \;ds.
\]
	By the continuity of the composition the curve
\[
	[0,1] \to H^{s-1}(\R^{2n};\R^{2n}),\;s \mapsto \nabla f(\varphi_1(x)+s(\varphi_2(x)-\varphi_1(x)),
\]
is continuous. Thus the above 
	\[
		f \circ \varphi_2-f \circ \varphi_1=\int_0^1 \nabla f \circ (\varphi_1 + s (\varphi_2-\varphi_1)) \cdot (\varphi_2-\varphi_1) \;ds
	\]
is an identity in $H^{s-1}$. Thus we conclude
	\[
		\|f \circ \varphi_2-f \circ \varphi_1\|_{H^{s-1}} \leq C \|\nabla f\|_{H^{s-1}} \|\varphi_2-\varphi_1\|_{H^{s-1}}.
	\]
This finishes the proof.
\end{proof}

As a corollary we have

\begin{Coro}\label{coro_lipschitz}
	Let $\varphi \in \Ds^s(\R^{2n})$. Then there is a neighborhood $W \subset \Ds^s(\R^{2n})$ of $\varphi$ and $C > 0$ s.t.
\[
	\|\varphi_2^{-1}-\varphi_1^{-1}\|_{H^{s-1}} \leq C \|\varphi_2-\varphi_1\|_{H^{s-1}}
\]
for all $\varphi_1,\varphi_2 \in W$.
\end{Coro}

\begin{proof}
	We write 
	\begin{align*}
		\varphi_2^{-1}-\varphi_1^{-1}=\left(\varphi_2^{-1} \circ \varphi_1 - \text{id}\right) \circ \varphi_1^{-1}=\left(\varphi_2^{-1} \circ \varphi_1 - \varphi_2^{-1} \circ \varphi_2\right) \circ \varphi_1^{-1}
	\end{align*}
	The continuity of $\Ds^s(\R^{2n}) \to \Ds^s(\R^{2n}),\psi \mapsto \psi^{-1}$ and Lemma \ref{lemma_lipschitz} ensures the existence of a neighborhood $W \subset \Ds^s(\R^{2n})$ of $\varphi$ s.t.
\begin{align*}
	\|\varphi_2^{-1}-\varphi_1^{-1}\|_{H^{s-1}} &= \|\left(\varphi_2^{-1} \circ \varphi_1 - \varphi_2^{-1} \circ \varphi_2\right) \circ \varphi_1^{-1}\|_{H^{s-1}} \\
	&\leq C \|\varphi_2^{-1} \circ \varphi_1 - \varphi_2^{-1} \circ \varphi_2\|_{H^{s-1}} \leq C \|\varphi_2-\varphi_1\|_{H^{s-1}}
\end{align*}
	for $\varphi_1,\varphi_2 \in W$. Note that $\varphi_2^{-1} \in \text{id}+H^s(\R^{2n};\R^{2n})$. This finishes the proof.
\end{proof}

For integer an integer $\sigma \geq 0$ one easily gets the estimate
\[
	\|f\|_{H^\sigma} + \|g\|_{H^\sigma} \leq C \|f+g\|_{H^\sigma}
\]
for $f,g \in H^\sigma(\R^{2n})$ with disjoint support and for some $C > 0$ independent of $f,g$. One just has to use the equivalent norm $\|f\|_{H^\sigma}=\left(\sum_{k=0}^\sigma \|D^k f\|_{L^2}^2\right)^{1/2}$. Here $D^k$ denotes the collection of all derivatives of order $k$. For noninteger $\sigma$ the Sobolev norm $\|\cdot\|_{H^\sigma}$ is defined in a nonlocal way. In this case we have the following replacement.

\begin{Lemma}\label{nonlocal_norm}
	Let $\sigma \geq 0$ and $m \geq 1$. Then there is $C > 0$ s.t. we have the following: Let $P_1,P_2 \in \R^m$ be with distance $0 < d:=|P_2-P_1| \leq 4$. Then 
\[
	\|f\|_{H^\sigma} + \|g\|_{H^\sigma} \leq C\|f+g\|_{H^\sigma},
\]
	for all $f,g \in H^\sigma(\R^m)$ with $\operatorname{supp}f \subset B_{d/4}(P_1)$ and $\operatorname{supp}g \subset B_{d/4}(P_2)$. Here $\operatorname{supp}$ denotes the support of a function and $B_r(P) \subset \R^m$ is the open ball of radius $r > 0$ around $P \in \R^m$.
\end{Lemma}

\begin{proof}
Since $\|\cdot\|_{H^\sigma}$ is translation and rotation invariant we can assume $P_1=(-d/2,0,\ldots,0),P_2=(d/2,0,\ldots,0)$. We denote by $K=\overline{B_3(0)}$ the closed ball around $0 \in \R^m$ with radius $5$. From \cite{bahouri} we know that the homogeneous Sobolev norm
	\[
		\|f\|_{\dot H^\sigma}=\left(\int_{\R^m} |\xi|^{2\sigma} |\hat f(\xi)|^2 \;d\xi \right)^{1/2},
	\]
	is equivalent to $\|\cdot\|_{H^s}$ for functions with support in $K$. If we denote by $f^\lambda(x)=f(\lambda x)$ the scaling of $f$ by $\lambda > 0$ we have $\|f^\lambda\|_{\dot H^\sigma}=\lambda^{\sigma-m/2} \|f\|_{\dot H^\sigma}$. If we take $\lambda=4/d$ then $\operatorname{supp} f^\lambda \subset B_1(Q_1)$ resp. $\operatorname{supp} g^\lambda \subset B_1(Q_2)$ where $Q_1=(-2,0,\ldots,0)$ resp. $Q_2=(2,0,\ldots,0)$. Now take $\kappa_1,\kappa_2 \in C_c^\infty(\R^m)$ s.t. $\kappa_1= 1$ on $B_1(Q_1)$ and $0$ on $B_{3/2}(Q_1)^c$ resp. $\kappa_2=1$ on $B_1(Q_2)$ and $0$ on $B_{3/2}(Q_2)^c$. We then have
	\begin{align*}
		\|f\|_{H^\sigma} + \|g\|_{H^\sigma} &\leq C (\|f\|_{\dot H^\sigma} + \|g\|_{\dot H^\sigma}) = C \lambda^{m/2-\sigma}  (\|f^\lambda\|_{\dot H^\sigma} + \|g^\lambda\|_{\dot H^\sigma})\\
		&\leq C \lambda^{m/2-\sigma}  (\|f^\lambda\|_{H^\sigma}+\|g^\lambda\|_{H^\sigma}) \\
		&= C \lambda^{m/2-\sigma}  (\|\kappa_1 (f^\lambda+g^\lambda)\|_{H^\sigma}+\|\kappa_2 (f^\lambda+g^\lambda)\|_{H^\sigma})\\
		&\leq C \lambda^{m/2-\sigma} \|f^\lambda+g^\lambda\|_{H^\sigma}
		\leq C \|f+g\|_{\dot H^\sigma} \leq C \|f+g\|_{H^\sigma},
	\end{align*}
	where we used that $\|\cdot\|_{\dot H^\sigma}$ and $\|\cdot\|_{H^\sigma}$ are equivalent for functions with support in $K$ and where we used the scaling for the homogeneous norm. This finishes the proof.
\end{proof}

Now we prove the special case $T=1$ of Theorem \ref{th_nonuniform}.

\begin{Prop}\label{prop_nonuniform}
	Let $n \geq 1$ and $s > 2n/2+1$. We denote by $\Phi$ the time $T=1$ solution map of \eqref{symplectic_euler}, i.e.
\[
	\Phi:H^s_\omega(\R^{2n};\R^{2n}) \to H^s_\omega(\R^{2n};\R^{2n}),\; u_0 \mapsto u(1;u_0).
\]
Then $\Phi$ is nowhere locally uniformly continuous.
\end{Prop}

\begin{proof}
	The proof is based on the conservation law \eqref{frozen}. For $t=1$ it reads as
	\begin{equation}\label{conserved}
		\nabla_\omega u(1;u_0)=(\nabla_\omega u_0) \circ \varphi(1;u_0)^{-1},\; u_0 \in H^s_\omega(\R^{2n};\R^{2n}).
	\end{equation}
	Note that $\varphi(1;u_0)=\exp(u_0)$. We fix $u_\ast \in S$ where $S$ is the dense subset from Lemma \ref{lemma_dexp}. With that we fix $w_\ast \in H^s_\omega(\R^{2n};\R^{2n})$, $x_\ast \in \R^{2n}$ and $m_\ast > 0$ s.t.
	\begin{equation}\label{m_ast}
		\|w_\ast\|_{H^s}=1 \text{ and }	|\left(d_{u_\ast} \exp(w_\ast)\right)(x_\ast)| = m_\ast  > 0.
	\end{equation}
We will construct $B_{R_\ast}(u_\ast) \subset H^s_\omega(\R^{2n};\R^{2n})$ for some $R_\ast > 0$ s.t. $\left. \Phi \right|_{B_R(u_\ast)},0 < R \leq R_\ast$ is not uniformly continuous. But we need some preparation. First we choose $R_1 > 0$ s.t.
	\begin{equation}\label{c1}
		\frac{1}{C_1} \|f \circ \varphi^{-1}\|_{H^{s-1}} \leq  \|f\|_{H^{s-1}} \leq C_1 \|f \circ \varphi^{-1}\|_{H^{s-1}}
	\end{equation}
	for all $f \in H^s(\R^{2n};\R^{2n})$ and $\varphi \in \exp(B_{R_1}(u_\ast))$ for some $C_1 > 0$. This is ensured by the continuity of the composition resp. inversion and the linearity in $f$. Next we choose $0 < R_2 \leq R_1$ s.t.
	\begin{equation}\label{c2}
	|\varphi(x)-\varphi(y)|\leq C_2 |x-y|,\;x,y \in \R^{2n},
	\end{equation}
	for all $\varphi \in \exp(B_{R_2}(u_\ast))$ for some $C_2 > 0$. This is clearly possible due to the Sobolev Imbedding Theorem $H^s(\R^{2n};\R^{2n}) \hookrightarrow C^1_0(\R^{2n};\R^{2n})$ and the continuity of $\exp$. Consider the Taylor expansion for $\exp:H^s_\omega(\R^{2n};\R^{2n}) \to \Ds^s(\R^{2n})$
\[
	\exp(p+h)=\exp(p)+d_{p}\exp(h)+\int_0^1 (1-t) d^2_{p + s h}\exp(h,h) \;ds
\]
	for $p,h \in H^s_\omega(\R^{2n};\R^{2n})$. Motivated by this we choose $0 < R_3 \leq R_2$ s.t.
	\begin{equation}\label{c3a}
		\|d_{p_1}^2\exp(h_1,h_2)\|_{H^s} \leq C_3 \|h_1\|_{H^s} \|h_2\|_{H^s}
	\end{equation}
resp.
	\begin{equation}\label{c3b}
	\|d_{p_1}^2\exp(h_1,h_2)-d_{p_2}^2\exp(h_1,h_2)\|_{H^s} \leq C_3 \|p_1-p_2\|_{H^s} \|h_1\|_{H^s} \|h_2\|_{H^s},
	\end{equation}
	for $p_1,p_2 \in B_{R_3}(u_\ast)$, $h_1,h_2 \in H^s_\omega(\R^{2n};\R^{2n})$ for some $C_3 > 0$. This is possible due to the smoothness of $\exp$. Again using the smoothness of $\exp$ we choose $0 < R_4 \leq R_3$ s.t. 
	\begin{equation}\label{c4}
		\|\exp(p_1)-\exp(p_2)\|_{H^s} \leq C_4 \|p_1-p_2\|_{H^s}
	\end{equation}
for $p_1,p_2 \in B_{R_4}(u_\ast)$ and for some $C_4 > 0$. Finally we choose $0 < R_\ast \leq R_4$ s.t.
	\begin{equation}\label{r_ast}
		\max\{C_3C_5 R_\ast,C_3C_5 R_\ast^2/4\} \leq m_\ast/16.
	\end{equation}
	For definiteness we fix $C_5 > 0$ s.t. we have for all $x \in \R^{2n}$
	\begin{equation}\label{c5}
		|w(x)| \leq C_5 \|w\|_{H^s},\;w \in H^s(\R^{2n};\R^{2n}).
	\end{equation}
	Now we take $0 < R \leq R_\ast$. Our goal is to show that $\left. \Phi \right|_{B_R(u_\ast)}$ is not uniformly continuous. We define the sequence of radii $(r_k)_{k \geq 1} \subset (0,\infty)$
	\begin{equation}\label{radii}
		r_k=\frac{m_\ast }{8kC_2},\; k \geq 1,
	\end{equation}
	and choose a sequence $(v_k)_{k \geq 1} \subset H^s_\omega(\R^{2n};\R^{2n})$ s.t. $\operatorname{supp}v_k \subset B_{r_k}(x_\ast)$ and $\|v_k\|_{H^s}=R/2$ for $k \geq 1$. One can construct this sequence by taking functions $H \in C_c^\infty(\R^{2n})$ and by considering $\frac{R}{2} \frac{\nabla_\omega H}{\|\nabla_\omega H\|_{H^s}}  \in H^s_\omega(\R^{2n};\R^{2n})$. With all this preparation we define the pair of sequences $(u_{0,k})_{k \geq 1},(\tilde u_{0,k})_{k \geq 1}$ of initial values
	\[
		u_{0,k}=u_\ast+v_k \text{ resp. } \tilde u_{0,k}=u_{0,k}+\frac{1}{k} w_\ast=u_\ast+v_k+\frac{1}{k}w_\ast,\; k \geq 1.
	\]
	Note that for some $N \geq 1$ we have $(u_{0,k})_{k \geq N},(\tilde u_{0,k})_{k \geq N} \subset B_R(u_\ast)$. Moreover
	\[
		\lim_{k \to \infty} \|\tilde u_{0,k}- u_{0,k}\|_{H^s}=\lim_{k \to \infty}  \|\frac{1}{k} w_\ast \|_{H^s}=0.
	\]
	To show nonuniform continuity on $B_R(u_\ast)$ it is sufficient to show 
	\[
		\liminf_{k \to \infty} \|\Phi(\tilde u_{0,k})-\Phi(u_{0,k})\|_{H^s} > 0.
	\]
To establish this it is clearly sufficient to show
	\[
		\liminf_{k \to \infty} \|\nabla_\omega \cdot \Phi(\tilde u_{0,k})-\nabla_\omega \cdot \Phi(u_{0,k})\|_{H^{s-1}} > 0.
	\]
In order to make the notation easier we introduce
	\[
	\tilde \varphi_k=\exp(\tilde u_{0,k}),\;	\varphi_k=\exp(u_{0,k}),\; k \geq N.
	\]
	Using \eqref{frozen} we have
\[
	\nabla_\omega \cdot \Phi(\tilde u_{0,k})=(\nabla_\omega \cdot \tilde u_{0,k}) \circ \tilde \varphi_k^{-1},\;\nabla_\omega \cdot \Phi(u_{0,k})=(\nabla_\omega \cdot u_{0,k}) \circ \varphi_k^{-1},\; k \geq 1.
\]
Thus we can write for $k \geq 1$
\begin{align*}
	&\nabla_\omega \cdot \Phi(\tilde u_{0,k})-\nabla_\omega \cdot \Phi(u_{0,k})=\left((\nabla_\omega \cdot u_\ast) \circ \tilde \varphi_k^{-1}-(\nabla_\omega \cdot u_\ast) \circ \varphi_k^{-1}\right)\\
	&+\left((\nabla_\omega \cdot v_k) \circ \tilde \varphi_k^{-1}-(\nabla_\omega \cdot v_k) \circ \varphi_k^{-1}\right)+\frac{1}{k} (\nabla_\omega \cdot w_\ast) \circ \tilde \varphi_k^{-1}.
\end{align*}
From \eqref{c1} we get
\[
	\|\frac{1}{k} (\nabla_\omega \cdot w_\ast) \circ \tilde \varphi_k^{-1}\|_{H^{s-1}} \leq C_1 \frac{1}{k} \|\nabla_\omega \cdot w_\ast\|_{H^{s-1}} \to 0
\]
as $k \to \infty$. From Lemma \ref{lemma_lipschitz}, Corollary \ref{coro_lipschitz} and \eqref{c4} we get
\begin{align*}
	&\|(\nabla_\omega \cdot u_\ast) \circ \tilde \varphi_k^{-1}-(\nabla_\omega \cdot u_\ast) \circ \varphi_k^{-1}\|_{H^{s-1}} \leq C \|\nabla_\omega \cdot u_\ast\|_{H^s} \|\tilde \varphi_k-\varphi_k\|_{H^s}\\
	&\leq C C_4 \|\nabla_\omega \cdot u_\ast\|_{H^s} \|\tilde u_{0,k}-u_{0,k}\|_{H^s} \to 0
\end{align*}
as $k \to \infty$ since $u_\ast \in H^{s+1}$. So using the reverse triangle inequality we have the reduction
\[
	\liminf_{k \to \infty} \|\nabla_\omega \cdot \Phi(\tilde u_{0,k})-\nabla_\omega \cdot \Phi(u_{0,k})\|_{H^{s-1}} = 
	\liminf_{k \to \infty} \|(\nabla_\omega \cdot v_k) \circ \tilde \varphi_k^{-1}-(\nabla_\omega \cdot v_k) \circ \varphi_k^{-1}\|_{H^{s-1}}.
\]
For the supports of the latter expressions we can ensure by \eqref{c2} and \eqref{radii}
\[
	\operatorname{supp}(\nabla_\omega \cdot v_k) \circ \tilde \varphi_k^{-1}  \subset B_{m_\ast/(8k)}(\tilde \varphi_k(x_\ast)) \text{ resp. } \operatorname{supp}(\nabla_\omega \cdot v_k) \circ \varphi_k^{-1}  \subset B_{m_\ast/(8k)}(\varphi_k(x_\ast)).
\]
In order to separate the supports we need to estimate $|\tilde \varphi_k(x_\ast)-\varphi_k(x_\ast)|$. To do so we consider
\begin{align*}
	&\tilde \varphi_k=\exp(u_\ast+v_k+\frac{1}{k}w_\ast)=\\
	&\exp(u_\ast)+d_{u_\ast} \exp(v_k+\frac{1}{k}w_\ast)+\int_0^1 (1-s) d_{u_\ast + s v_k + s\frac{1}{k}w_\ast}^2 (v_k+\frac{1}{k}w_\ast,v_k+\frac{1}{k} w_\ast) \;ds
\end{align*}
and
\begin{align*}
&\varphi_k=\exp(u_\ast+v_k)=\exp(u_\ast)+d_{u_\ast} \exp(v_k)+\int_0^1 (1-s) d_{u_\ast + s v_k}^2 (v_k,v_k) \;ds.
\end{align*}
Thus we get
\[
	|\tilde \varphi_k(x_\ast)-\varphi_k(x_\ast)|=\left|\frac{1}{k} (d_{u_\ast}\exp(w_\ast))(x_\ast)+I_1(x_\ast)+I_2(x_\ast)+I_3(x_\ast)\right|,
\]
where
\[
	I_1=\int_0^1 (1-s) \left(d_{u_\ast + s v_k + s\frac{1}{k}w_\ast}^2 (v_k,v_k)- d_{u_\ast + s v_k}^2 (v_k,v_k)\right)\;ds
\]
and
\[
	I_2=2 \int_0^1 (1-s) d_{u_\ast + s v_k + s\frac{1}{k}w_\ast}^2 (v_k,\frac{1}{k}w_\ast)\;ds
\]
and
\[
	I_3=2 \int_0^1 (1-s) d_{u_\ast + s v_k + s\frac{1}{k}w_\ast}^2 (\frac{1}{k}w_\ast,\frac{1}{k}w_\ast)\;ds.
\]
Using \eqref{c3a} we conclude 
\[
	\|I_1\|_{H^s} \leq C_3 \frac{1}{k} \|w_\ast\|_{H^s} \|v_k\|_{H^s}^2 = C_3 \frac{1}{k} R^2/4
\]
and using \eqref{c3b} we get
\[
	\|I_2\|_{H^s} \leq C_3 \frac{2}{k} \|w_\ast\|_{H^s} \|v_k\|_{H^s}=C_3 \frac{1}{k} R,\;
	\|I_3\|_{H^s} \leq C_3 \frac{1}{k^2}.
\]
By the condition for $R_\ast$ in \eqref{r_ast} we get by \eqref{c5} 
\[
	|I_1(x_\ast)| \leq m_\ast/(16k),\;|I_2(x_\ast)| \leq m_\ast/(16k).
\]
For $k \geq N$ large enough we get $|I_3(x_\ast)| \leq m_\ast/(16k)$. Hence 
\[
	|\tilde \varphi_k(x_\ast)-\varphi_k(x_\ast)|=\frac{1}{k}|(d_{u_\ast}\exp(w_\ast)(x_\ast)|-|I_1(x_\ast)+I_2(x_\ast)+I_3(x_\ast)| \geq m_\ast/(2k)
\]
for $k$ large enough. The ratio of the radius of the supports $m_\ast/(8k)$ to the distance $|\tilde \varphi_k(x_\ast)-\varphi_k(x_\ast)| \geq m_\ast/(2k)$ allows us to use Lemma \ref{nonlocal_norm} to conclude with \eqref{c1}
\begin{align*}\liminf_{k \to \infty} \|(\nabla_\omega \cdot v_k) \circ \tilde \varphi_k^{-1}-(\nabla_\omega \cdot v_k) \circ \varphi_k^{-1}\|_{H^{s-1}} \geq \liminf_{k \to \infty} \frac{2}{CC_1} \|\nabla_\omega \cdot v_k\|_{H^{s-1}}.
\end{align*}
Note that $\|f\|_{L^2}+\|\nabla f\|_{H^{s-1}}$ is equivalent to $\|f\|_{H^s}$ for $f \in H^s$. In other words there is $K_1 > 0$ s.t.
\[
	\|f\|_{H^s} \leq K_1 (\|f\|_{L^2}+\|\nabla f\|_{H^{s-1}}),\;f \in H^s(\R^{2n};\R^{2n}).
\]
From Lemma \ref{symplectic_divergence} we get that
\[
	\|\nabla v_k\|_{H^{s-1}} \leq K_2 \|\nabla_\omega \cdot v_k\|_{H^{s-1}}
\]
for some $K_2 > 0$. Since the support of $v_k$ shrinks to a point as $k \to \infty$ and $\|v_k\|_{H^s}=R/2$ is bounded we get 
\[
	\lim_{k \to \infty} \|v_k\|_{L^2}=0.
\]
Thus 
\[
	\liminf_{k \to \infty} \|\nabla_\omega \cdot v_k\|_{H^{s-1}} \geq \frac{1}{K_1 K_2}  \liminf_{k \to \infty} \|v_k\|_{H^s} = \frac{R}{2 K_1 K_2}.  
\]
Putting all together gives
\[
	\liminf_{k \to \infty} \|\nabla_\omega \cdot \Phi(\tilde u_{0,k})-\nabla_\omega \cdot \Phi(u_{0,k})\|_{H^s} \geq \frac{R}{CC_1K_1K_2} > 0.
\]
So
\[
\liminf_{k \to \infty} \|\Phi(\tilde u_{0,k})-\Phi(u_{0,k})\|_{H^s} > 0
\]
whereas $\lim_{k \to \infty} \|\tilde u_{0,k}-u_{0,k}\|_{H^s}=0$. This shows that $\left. \Phi \right|_{B_R(u_\ast)}$ is not uniformly continuous. Since $0 < R \leq R_\ast$ and $u_\ast \in S$ is arbitrary this finishes the proof.
\end{proof}

Now we can prove Theorem \ref{th_nonuniform}.

\begin{proof}[Proof of Theorem \ref{th_nonuniform}]
	The proof follows from Proposition \ref{prop_nonuniform} and \eqref{solution_map}.
\end{proof}

\bibliographystyle{plain}

\end{document}